\documentclass{article}[12pt]
\usepackage[a4paper, total={6in, 8in}]{geometry}

\usepackage{amsmath, amsthm, amssymb, amsfonts, mathtools}
\usepackage{enumerate}
\usepackage{mathrsfs}

\newcommand{\norm}[1]{\lVert#1\rVert}
\newcommand{\la}{\lambda}
\newcommand{\bC}{{\mathbb C}}
\newcommand{\cH}{{\mathcal H}}
\newcommand{\bH}{{\bf H}}
\newcommand{\cN}{{\mathcal N}}
\newcommand{\cM}{{\mathcal M}}
\newcommand{\cK}{{\mathcal K}}
\newcommand{\cW}{{\mathcal W}}
\newcommand{\cJ}{{\mathcal J}}
\newcommand{\mD}{{\mathscr D}}

\usepackage[colorinlistoftodos]{todonotes}

\newtheorem{theorem}{Theorem}
\newtheorem{lemma}{Lemma}
\newtheorem{proposition}{Proposition}
\newtheorem{corollary}{Corollary}

\newtheorem{definition}{Definition}

\newtheorem{examples}{Examples}
\theoremstyle{remark}

\newtheorem{remark}{Remark}

\author{Alexandru Aleman, Alex Bergman}               
\title{Invariant Subspaces for Generalized Differentiation and Volterra Operators}
\date{\today}

\begin{document}
	\maketitle
    \begin{abstract}
        In this paper we provide a far-reaching generalization of the existent results about invariant subspaces of the differentiation operator $D=\frac{\partial}{\partial t}$ on $C^\infty(0,1)$ and the Volterra operator $Vf(t)=\int_0^tf(s)ds$, on $L^2(0,1)$.  We use an abstract approach to study invariant subspaces of pairs $D,V$ with $DV=I$,  where $V$ is compact and quasi-nilpotent and $D$ is unbounded densely defined and closed on the same Hilbert space. Our results cover many differential operators, like Schr\"odinger operators and a large class of other canonical systems, as well as the so-called compact self-adjoint operators with removable spectrum recently studied in \cite{MR3292735}. Our methods are based on a model for such pairs which involves de Branges spaces of entire functions and plays a crucial role in the development. However, a number of difficulties arise from the fact that our abstract operators  do not necessarily identify with the usual operators on such spaces, but  with rank one perturbations of those, which, in terms of invariant subspaces creates a number of challenging problems.
    \end{abstract}
    
	\begin{section}{Introduction}
		The motivation for the study initiated in this paper goes back to the invariant subspaces of the Volterra operator $V$ defined on the Hilbert space $L^{2}(0,1)$ by $$Vf(t)=\int_0^tf(s)ds,$$	
		as well as the invariant subspaces of its unbounded left inverse 
		\begin{equation*}
			Df(x) = f'(x),
		\end{equation*}
	when restricted to $C^\infty(0,1)$.
		In 1938 Gelfand \cite{gelfand} posed the problem of determining the invariant subspaces of $V$. This problem has a rich history and has been studied by many authors.
		The, by now well known result,  that each such subspace consists of functions which vanish a.e. on a fixed interval $(0,a)$,  has been originally proved by Agmon, \cite{MR0031110} and independently by Brodskii \cite{MR0094717} and Donoghue \cite{MR0092124}. Later, among others, Sarason \cite{MR0192355} gave a different proof based on properties of the backward shift operator on the Hardy space on the upper half plane.
		
		The study of invariant subspaces of $D$ on $C^\infty(0,1)$ has been initiated much later by the first author and B. Korenblum   (\cite{MR2419491}, 2008). The structure of these subspaces is much more complicated. They can be classified according to the spectrum of the restriction of $D$ which can be either void, or a discrete set in $\bC$, or the entire complex plane. The first category, also called \emph{residual subspaces} is fully described in 	\cite{MR2419491} and are of primary interest for this work, It turns out, that a residual subspace consists of functions in $C^\infty(0,1)$ that vanish on a compact subinterval. This interval may reduce to a point in which case we have that the functions  together with all their derivatives vanish at that point. This is clearly related to the invariant subspaces of right inverses of $D$, that is, Volterra operators of the type $$V_af(t)=\int_a^tf(s)ds, \quad a\in [0,1].$$		The case of differentiation invariant subspaces where the spectrum of the restriction is discrete is related to the appropriate version of \emph{spectral synthesis} in this context. This subtle problem was investigated by the first author, A. Baranov and Y. Belov in 
	\cite{MR3318654} and completely  solved subsequently by A. Baranov and Y.  Belov in \cite{MR3925104}. The structure of the remaining type of $D-$invariant subspaces is not understood.
		
		The ideas related to invariant subspaces of the Volterra operator have been extended to other classes of integral operators. The closest to the subject of this paper is the direction suggested by C. Remling in his book \cite{MR3890099}, namely to describe the invariant subspaces of Green's operators associated to canonical systems, for example,  Schr\"odinger operators. We shall provide a description of these objects in Section \ref{sec:new_deBranges_results}. The reason for our particular interest is the interplay between  invariant subspaces  of these operators and the so-called residual subspaces of their left inverses which, in this case, are differential operators. 
		
		It is the aim of this paper to study such situations in a very general context. More precisely, we are interested in invariant subspaces for pairs of Hilbert space operators $D,V$ satisfying $DV=I$ together with certain abstract conditions  which mimic the situations described above and cover a large class of examples of interest.
	The formal definition is as follows.
	
		\begin{definition}
			Let $H$ be a separable Hilbert space. A pair of operators $D$ and $V$ will be called \emph{admissible} if
			\begin{enumerate}[(i)]
				\item $D : \mathscr{D}(D) \subset H \to H$ is a densely defined closed  operator,
				\item  $D$ has a one-dimensional kernel
				\item $D$ has a self-adjoint restriction with compact resolvent,
				\item $V: H \to H$ is compact and quasi-nilpotent, that is, $\sigma(V) = \left\{ 0 \right\}$,
				\item $\text{Ran}(V) \subset \mathscr{D}(D)$ and $DV = I$.
			\end{enumerate}
		\end{definition}
	 The assumption that $V$ is compact is actually redundant, since it follows from the compactness of the resolvent of the self-adjoint operator in (iii). 
	 
	 In this setup the operator $V$ plays the role of the Volterra operator and $D$ of the derivative multiplied by appropriate constants to ensure (iii). Other typical examples of pairs of admissible operators are given by second order differential operators with one fixed boundary condition and their Green's operators (solution operators). We give a large class of examples including Schrödinger operators and canonical systems of differential equations in Section \ref{sec:examples}. Finally, these conditions cover also the class of self-adjoint operators with removable spectrum  introduced  in \cite{MR3292735}.  These are compact self-adjoint operators having a rank one perturbation which is quasi-nilpotent. We shall discuss in detail these examples in  Section \ref{sec:examples}.
				
		The key tool in our investigation is a functional model which is obtained in a similar way  to the approaches in spectral theory. More precisely, it turns out that $D$ possesses a family $\{\phi_\lambda:~\lambda\in \bC\}$ of eigenvectors such that the function $\lambda\mapsto \phi_\lambda$ is an entire $H-$valued function, so that the \emph{generalized Fourier transform}  $\cW$ defined on $H$ by $$\cW x(\la)=\langle x,\phi_{\overline{\la}}\rangle,\quad x\in H$$
		induces a unitary map onto a Hilbert space of entire functions. Our first main theorem identifies this space as a shifted de Branges space, that is,  
		\begin{equation}\label{intro_model}\cW H=e^{-i\alpha z}\cH(E),\end{equation}
		where $\alpha\in \mathbb{R}$,  $E$ is a regular Hermite-Biehler function  and $\cH(E)$ is the corresponding de Branges space of entire functions (see Section \ref{sec:prelim} for details). Here and throughout by the usual abuse of notation we write $z$ for both a point in $\bC$ as well as for the identity function on $\bC$.
		While in spectral theory of certain differential operators such models yield a ''diagonalization'', in our case the situation is more complicated; $\cW D\cW^{-1}$ is just an extension of $M_z^*$,  where $M_z$ is multiplication by the independent variable and $\cW V^{*}\cW^{-1}$ is a rank one  perturbation of the backward shift $f\mapsto\frac{f-f(0)}{z}$.

        De Branges spaces have attracted renewed attention is recent years and have been an instrumental tool in resolving many open problems in operator theory, complex analysis, harmonic analysis, and spectral theory, see e.g. \cite{MR4635834, MR4307215, MR2739783, MR1388849, MR1915824, MR2609247, MR4639943, MR2628803, MR3425390, MR1923965, MR2910798, MR3561432} for some examples of this.
	
		The fact that the lattice of invariant subspaces of the Volterra operator is totally ordered by inclusion has been extended in many ways. Such operators are called \emph{unicellular}. Kalisch \cite{MR0091443} and Kisilevskiĭ \cite{MR0221323} proved unicellularity of a large class of integral operators whose kernels do not vanish on the diagonal.  Another interesting class consists of operators with positive imaginary part.
		Kisilevskiĭ \cite{MR0221323} proved that a  cyclic irreducible operator $V$ with  positive and nuclear imaginary part is unicellular. A very good account for results of this type is the book of Gohberg and Krein \cite{MR0264447}. We should point out that none of the conditions above apply to the operators $V$ considered in this paper.
		
	Clearly, $D$ has many right inverses which are all compact since by assumption (ii) above they differ by a rank one operator. In Section \ref{sec:new_deBranges_results} we shall show that among these, the quasi-nilpotent ones form a one-parameter family $V_\beta, ~|\beta+\alpha|\le \tau(E)$. They generalize the Volterra-type operators $$f\mapsto \int_\beta^tf(s)ds,\quad \beta\in [0,1],$$
	which are unicellular if and only if $\beta=0$, or $\beta=1$.
	It turns out that a similar result holds in the general context as well. This is due to the fact that the invariant subspaces of these operators can be described in terms of the model space. More precisely, their ''Fourier transforms'' have the form 	 
	$$\cW M=(e^{i\gamma z}\cH(F))^\perp,\quad \gamma\in \mathbb{R},$$
where the de Branges space	 $\cH(F)$ is a closed subspace  of $\cH(E)$.  These subspaces always form a chain. Moreover for fixed $\beta$ as above, $\gamma$ must satisfy
	\begin{equation}\label{v_beta_star_inv_intro}
		|\gamma+\alpha|\le \tau(E)-\tau(F),\quad 	|\beta-\gamma|\le \tau(F).
	\end{equation} 
	This leads to a general unicellularity theorem.		For an entire function $f$ we denote by $\tau(f)$ its exponential type.
	
		\begin{theorem}\label{thm:1'}
			Let $D$ and $V$ be an admissible pair. Then $D$ has at most two (and at least one) unicellular quasi-nilpotent right-inverses given by
			\begin{equation*}
			V_{+}f = Vf + \langle f, \mathcal{W}^{-1}s_{+} \rangle \phi_{0},
			\end{equation*}
			and
			\begin{equation*}
			V_{-}f = Vf + \langle f, \mathcal{W}^{-1}s_{-} \rangle\phi_{0},
			\end{equation*}
			where
			\begin{equation*}
			s_{\pm} = \frac{1-e^{i\beta_{\pm}z}}{z} \in \mathcal{W}H,
			\end{equation*}
			and
			\begin{equation*}
			\beta_{+} = \tau(E) - \alpha \text{, and } \beta_{-} = - (\tau(E) + \alpha).
			\end{equation*}
			These operators coincide if and only if $\tau(E) = \alpha = 0$.
		\end{theorem}
		In particular, the original operator $V$ is unicellular if and only  if $|\alpha|=\tau(E)$.\\
	These general results have a number of applications 
	which are presented in \S \ref{some_applications}. 
These 	include the result mentioned above regarding invariant subspaces of Green's operators associated to canonical systems, as suggested  by C. Remling in \cite{MR3890099}.
It turns out, that for singular Hamiltonians these are again given by functions that vanish on a.e. on interval $(0,c)$. The non-singular case is more involved. Here we show that every such Green's operator has a rank one perturbation which is a Green's operator as well and is unicellular.\\
 The unicellularity result applies to compact self-adjoint operators with removable spectrum as defined above. We show that   under some natural assumptions, such operators possess a quasi-nilpotent rank one perturbation which is also
 unicellular. This completes the recent work of Baranov and Belov \cite{MR3292735} on the subject.

		The more difficult problem within this circle of ideas is to describe the invariant subspaces of $D$. The first step is to restrict $D$ to an appropriate space  where it acts as a continuous operator. This is the natural  analogue of $C^\infty$ defined by 
				\begin{equation*}
		C^{\infty}(D) = \bigcap_{n \geq 0} \mathscr{D}(D^{n})
		\end{equation*}
	which becomes a Fr\'echet space with respect to the translation invariant metric
		\begin{equation*}
		d(f,g) = \sum_{j} 2^{-j} \frac{\norm{D^{j}f-D^{j}g}_{H}}{1+\norm{D^{j}f-D^{j}g}_{H}}.
		\end{equation*}
	As in the classical case, we prove that the invariant subspaces of $D$ split naturally into three categories according to the spectrum of the restriction of $D$. Theorem \ref{thm:spectrum} shows that if $\cJ\subset C^\infty(D)$ is closed and $D-$invariant then $\sigma(D|\cJ)$ is either void, or a discrete set of eigenvalues, or the whole complex plane.

		In this paper we focus on closed $D-$invariant subspaces  $\cJ\subset C^\infty(D)$ with $\sigma(D|\cJ)=\emptyset$. These subspaces are called \emph{residual}. Our approach differs essentially from the approach in \cite{MR2419491}, which built heavily on the self-adjoint  operator of multiplication by independent variable, $Mf(x) = xf(x)$ 
		 on $L^2(a,b)$. A possible analogue in the  general context  would be a 
		solution of the Heisenberg commutation relation $$DM-MD=I. $$
		Unfortunately, it turns  out (see for example \cite{MR689997, MR716970} or Chapter 12 of \cite{MR1192782}) 
		that the admissible pairs for which the commutator equation has a self-adjoint solution are essentially unitarily equivalent to the classical one. Thus this 		
	 technique is not available to us and instead our approach is based on the theory of nearly invariant subspaces of de Branges spaces discussed  in Section \ref{sec:prelim}.
		
	Intuitively, the description of residual $D-$invariant subspaces  is related to finding the appropriate analogue of the space  of $C^\infty-$functions vanishing on a nontrivial interval, or vanish at one point together with all their derivatives. The idea developed in this paper is to describe such subspaces using  Volterra-invariance. Thus:\\
	1)   $C^\infty-$ functions vanishing on an  interval can be interpreted as
	$$\cJ_M=\{x\in C^\infty(D):~ D^jx\in M\},$$
	where $M\subset H$ is invariant for a quasi-nilpotent right inverse of $D$, In this case, by \eqref{v_beta_star_inv_intro} there exists a compact interval $J\subset [-\alpha-\tau(E), -\alpha+\tau(E)]$, such that $M$ is invariant for all operators $V_\beta,\beta\in J$.\\
	2) $C^\infty-$functions vanishing together with all their derivatives at a given point can be interpreted as 
	$$\cJ_M=\{x\in C^\infty(D):~ D^jx\in M\},$$
	where this time $M=V_\beta C^\infty(D)$  for some fixed $\beta\in [-\alpha-\tau(E), -\alpha+\tau(E)]$.
	
	Remarkably enough this leads to a complete description of residual subspaces in this quite general context.
	\begin{theorem}\label{1st_residual} Every $D-$invariant residual subspace $\cJ$ of $C^\infty(D)$ is of the form $\cJ_M$. The subspace $\cJ$ has the form 1) above if and only if $\cJ$ is not dense in $H$. In this case, the subspace $M\subset H$ is the closure of $\cJ$ in $H$.	
\end{theorem}
		A very natural and basic question which arises is whether subspaces of the form 1) are intersections of subspaces of form 2). More precisely, if $\cJ_M$ is of the form 1) with corresponding interval $J$, it is true 	that 
	$$\cJ_M=\bigcap_{\beta\in J}V_\beta C^\infty(D)\,?$$
	The answer is negative in this generality, and the reason is that the representation $\cJ=\cJ_M$ is not unique despite the fact that the interval $J$ in 1) is uniquely determined by $\cJ$. 
	This is a quite subtle phenomenon which we discuss is discussed at the end of  \S \ref{annihilator_residual}.
	 However, in Theorem \ref{zero-based} we prove that given an interval $J\subset [-\alpha-\tau(E), -\alpha+\tau(E)]$, with nonvoid interior, 
	$$\cJ_0=\bigcap_{\beta\in J}V_\beta C^\infty(D),$$
	is the largest residual subspace corresponding to this interval. In the standard case this is also the unique such subspace and it turns out that this continues to hold for other cases as well, more precisely, when the 	chain of de Branges subspaces of $\cH(E)$ is \emph{thin} in the sense of Belov and Borichev  \cite{MR4507623}.\\
	The opposite situation occurs when the  Hermite-Biehler function $E$ in our model has exponential type zero. In this case all intervals above reduce to $\{0\}$ and in Corollary \ref{tau=0}
	we prove that the set of residual subspaces is totally ordered.
	For example (see Corollary \ref{schrodinger_last}), this implies that for any regular Schr\"odinger operator $D$ on $[a,b]$ with a separated boundary condition at $a$, a residual subspace consists of functions that vanish on $[a,c],~a<c<b$, or  when $c=a$, it has the form described in 2) with $\beta=a$.	
		
The paper is organized as follows. Section \ref{sec:examples} contains a list of examples of admissible pairs which will be used to illustrate the general results. Section \ref{sec:prelim}
contains preliminary material  regarding de Branges spaces of entire functions and their connection to canonical systems. 	Section	\ref{sec:model} regards the model associated to an admissible pair. Unicellularity of generalized Volterra operators and their invariant  subspaces are considered in Section \ref{sec:new_deBranges_results}.  Section
\ref{sec:D_invariant} regards the invariant subspaces of $D$, especially  the structure of residual subspaces.

	\end{section}
	
	\begin{section}{Examples}\label{sec:examples}
		We begin by providing a large class of examples of admissible pairs $D$ and $V$.
		\subsection{The classical case}
		Let $D = -i\frac{d}{dx}$ be defined on the absolutely continuous functions on $(0,1)$,
        \begin{equation*}
            \mathscr{D}(D) = \left\{f \in AC(0,1): f, Df \in L^{2}(0,1) \right\}. 
        \end{equation*}
        It is well known that $D$ has simple eigenvalues at each point in the plane with eigenfunction $e^{i\lambda x}$ and self-adjoint restrictions with compact resolvent. Also it is clear that the classical Volterra operator
		\begin{equation*}
			Vf(x) = i\int_{0}^{x} f(t)dt,
		\end{equation*}
		is a quasi-nilpotent right inverse of $D$.
		
		\subsection{Regular Schr\"odinger operators}\label{Schr_reg}
		
		Let $-\infty < a < b < \infty$ and consider the Schrödinger operator $D = -\frac{d^{2}}{dx^{2}} + q$, where $q \in L^{1}(a,b)$ is real-valued. For the facts to come we refer the reader to any one of the references \cite{MR216050, MR566954, MR3243083}. The papers \cite{MR1943095, MR2215727, MR4033521} also treat Schrödinger equations with a similar flavor to our work. For $f \in L^{2}(a,b)$ consider the Cauchy problem
		\begin{equation}\label{eq:cauchyproblem}
			\begin{cases}
				(D-\lambda)u = f \\
				u(a) = 0 \\
				\dot{u}(a) = 0
			\end{cases}.
		\end{equation}
		By a solution to \eqref{eq:cauchyproblem} we shall mean a function $u$ belonging to the maximal domain of $D$
		\begin{equation*}
			\mathscr{D}(D) = \left\{ u \in \text{AC}(a,b) : \dot{u} \in \text{AC}(a,b) \text{ and } u, -\ddot{u}+qu \in L^{2}(a,b) \right\},
		\end{equation*}
		satisfying $(D-\lambda)u = f$ pointwise a.e. on $(a,b)$ and $u(a) = \dot{u}(a) = 0$. It is well known that the Cauchy problem has a unique solution $u_{\lambda} \in \mathscr{D}(D)$ for each $\lambda \in \mathbb{C}$, $f \in L^{2}(a,b)$ and $\norm{u_{\lambda}}_{2} \leq C_{\lambda}\norm{f}_{2}$, where $C_{\lambda} > 0$ is a constant not depending on $f$. Also the map $\lambda \mapsto u_{\lambda}(t)$ is entire for each fixed $t$. In particular the map $V_{\lambda}f = u_{\lambda}$ defines a bounded linear operator, such that $(D-\lambda)V_{\lambda}f = f$ for all $f \in L^{2}(a,b)$. We let $V = V_{0}$. The operator $V_{\lambda}$ is a Hilbert-Schmidt integral operator. Indeed if $v_{\lambda}$ and $u_{\lambda}$ are any two linear independent solutions and $W(u_{\lambda}, v_{\lambda})$ is their Wronskian, then
		\begin{equation}\label{eq:Green_operator}
			V_{\lambda}f(x) = \frac{1}{W(u_{\lambda},v_{\lambda})}\int_{a}^{x} \left( u_{\lambda}(x)v_{\lambda}(t) - u_{\lambda}(t)v_{\lambda}(x) \right)f(t)dt.
		\end{equation}
		Note that the integral kernel vanishes on the diagonal. It is not difficult to see that $V$ is quasi-nilpotent. By separated boundary condition $(\alpha)$, $0 \leq \alpha < \pi$, at $a$ we mean the functions in the maximal domain $\mathscr{D}(D)$ satisfying
		\begin{equation*}
			u(a)\cos(\alpha) - \dot{u}(a)\sin(\alpha) = 0.
		\end{equation*}
		We define separated boundary conditions at $b$ analogously. Selecting any separated boundary condition at $a$ gives an admissible pair $D$, $V$.

		\subsection{Canonical systems}\label{sec:canonical_systems}
		The following class of examples provides a far-reaching generalization of the regular Schr\"odinger operators from above.
        
		Fix $0 < \ell < \infty$. Let $\bH : (0,\ell) \to \mathbb{R}^{2\times 2}$ be a positive, $\bH \geq 0$, real-matrix valued function on $(0,\ell)$. To simplify the exposition we assume that $\bH$ has constant trace equal to $1$ on $(0,\ell)$. A canonical system is a differential equation of the form
		\begin{equation}\label{eq:canonical_system}
		\Omega \partial_{t} X(t) = \lambda \bH(t)X(t), \; \; \; \; \; \Omega = \begin{pmatrix} 0 & -1 \\ 1 & 0\end{pmatrix}, \; \lambda \in \mathbb{C}.
		\end{equation}
		The matrix-valued function $\bH$ is called the \emph{Hamiltonian} of the canonical system. 
		
		Let $L^{2}((0, \ell);\bH)$ be the weighted $L^2-$space on $(0,\ell)$ with norm
		\begin{equation*}
			\norm{f}_{L^{2}((0, \ell);H)}^{2} = \int_{0}^{\ell} \langle H(x)f(x), f(x) \rangle_{\mathbb{C}^{2}} dx < \infty,
		\end{equation*}
		where we identify elements $f,g$ with  with $\norm{f-g}_{L^{2}((0,\ell);\bH)} = 0$.
		
		In order to construct admissible pairs, the simplest case is when $\det(\bH)\ne 0$ a.e.. In this case, the unbounded operator $D$ given by 
		$$D=\bH^{-1}\Omega\frac{d}{dt}$$
		with maximal domain of definition is closed.  Self-adjoint restrictions with compact resolvent can be easily found using boundary conditions.\\
		Now define $V$ by
		\begin{equation}\label{def_V}
			Vf(x) = -\Omega \int_{0}^{x} \bH(t)f(t)dt, f \in \mathscr{H},
					\end{equation} 
and note that it satisfies $DV=I$.	
\begin{proposition}\label{V_canonical} The operator $V$ defined by \eqref{def_V} is compact and quasi-nilpotent on $L^{2}((0, \ell);\bH)$.
	\end{proposition}	
\begin{proof}
Since $\text{trace}(\bH)=1$ we have $\bH(x)\le I$ and
	$$\|Vf\|\le \int_0^x \sqrt{\langle \bH f(t), f(t)\rangle }dt.$$
With a standard argument	we conclude for any bounded sequence $(f_n)$ in  $L^{2}((0, \ell);\bH)$, that $(Vf_n)$ is uniformly bounded and equicontinuous, hence by the Arzela-Ascoli theorem, it has a uniformly subsequence. This implies the compactness of $V$.  \\ To see that  $V$ is quasi-nilpotent it suffices to check that $V$ has no eigenvalues. Obviously, $0$ is not an eigenvalue.  If $(I-\lambda V)f = 0$, note that  both $0$ and $f$ are solutions to the Cauchy problem
	\begin{equation*}
		\begin{cases}
			\Omega X' = \lambda \bH X \\
			X(0) = 0
		\end{cases}
	\end{equation*}
which gives $f=0$ a.e. and completes the proof.\end{proof}	
The case when $\bH$ is not invertible a.e. is  more involved because it requires an appropriate modification of the definition of $D$, which will then lead to a modified definition of the underlying Hilbert space. Here we shall only give a brief overview of the necessary steps and refer the reader to Section 2 of \cite{https://doi.org/10.48550/arxiv.1408.6022} or the first few chapters of \cite{MR3890099} for a detailed account.	Let us note from the beginning that if $\bH$ is singular on a set of positive measure any element $f\in L^{2}((0, \ell);\bH)$ is identified with its projection on the range of $\bH$.

We start with	the following technical definition which plays an important role.			
\begin{definition}
			A singular interval is an interval $I = (a,b) \subset (0, \ell)$, such that $\bH(x) = \langle \cdot, e_{\alpha} \rangle e_{\alpha}$, a.e. for some $e_{\alpha} = (\cos(\alpha), \sin(\alpha))^{T}$ and $(a,b)$ is contained in no larger interval where the same is true. A point $t \in (0,\ell)$ is called singular if it belongs to the interior of a singular interval, otherwise it is called regular.
		\end{definition}
Let us consider  the closed subspace 	$\mathscr{H}_\bH$ of $L^{2}((0,\ell);H)$ consisting of functions  which are constant a.e. on singular intervals,	that is, 	
	\begin{equation*}
		\mathscr{H}_\bH = \bigcap_{I\text{ singular}}\left\{ f \in L^{2}((0,\ell);\bH) : \bH f|I  \text{ is constant a.e.} \right\}.
	\end{equation*}
Next we define the domain of $D$  as the subspace 	$\mathscr{D}(D)\subset \mathscr{H}_\bH$ consisting of $f\in \mathscr{H}_\bH$ with 
 \begin{enumerate}[(i)]
		\item $f$  is absolutely continuous on  $(0,\ell)$, 
	\item  there exists  $g \in \mathscr{H}$ with $\Omega f' = Hg$, 
	\item  $\langle f(0),(0,1)^{T} \rangle = 0.$
	\end{enumerate}
	The main problem which arises here is that this domain in not necessarily dense in $\mathscr{H}_\bH$. However, one can address this difficulty with help of some natural assumptions. A straightforward modification of the arguments in the proofs of Corollary 2, Theorem 2 in Section 2 of \cite{https://doi.org/10.48550/arxiv.1408.6022} (by just omitting one boundary condition) leads to the following result.
	\begin{theorem}\label{D_def_sing}
	If $\bH$ satisfies 
	\begin{enumerate}[(i)]
		\item There exists no $\epsilon > 0$, such that $H(x) = \langle \cdot, e_{\pi/2} \rangle e_{\pi/2}$ a.e on $(0,\epsilon)$, with $e_{\pi/2} = (0,1)^{T}$.
		\item There exists no $\epsilon > 0$, such that $(\ell-\epsilon, \ell)$ is a singular interval,
	\end{enumerate}
	then $\mathscr{D}(D)$ is dense in $\mathscr{H}_\bH$ and the operator $D:\mathscr{D}\mapsto \mathscr{H}_\bH$ by $Df=g$, with $g$ given by (ii) is well-defined, closed and possesses a self-adjoint restriction.	
	\end{theorem}
Finally, note that $\mathscr{H}_\bH$ is invariant for $V$. Indeed, let $I$ be a singular interval and let $f\in \mathscr{H}_\bH$. Then 
since $\Omega e_\alpha\in \ker\bH(x),~x\in I$, it is easily seen that $Vf\in \mathscr{H}_\bH$. Then Proposition \ref{V_canonical} shows that $D,V$ is an admissible pair.
		
		\subsection{Self-adjoint operators with removable spectrum}\label{sec:compact_A}
		
		Let $A$ be an injective compact self-adjoint operator with simple removable spectrum, i.e. there exists $x,y \in H$, such that the rank one perturbation
		\begin{equation}\label{qn_pert}
		Vu = Au + \langle u, x \rangle y, \quad u \in H,
		\end{equation}
		is  quasi-nilpotent. Such operators were studied in \cite{MR3292735}. Let us assume that $y\notin \text{Ran}(A)$ and define 
		$\mathscr{D}(D) = \text{Ran}(A) +\bC y$ and 
		\begin{equation}\label{def_D_sa_case}
			D : \mathscr{D}(D) \subset H \to H, \; \; D(Au+cy) = u.
		\end{equation}
		
	\begin{proposition}\label{removable_spectrum}  Let $A$ be be an injective compact self-adjoint operator with simple removable spectrum, let $V$ be the quasi-nilpotent operator  given by \eqref{qn_pert} and assume that $y\notin \text{Ran}(A)$. If $D$ is given by \eqref{def_D_sa_case}, then $D,V$ is an admissible pair.
	\end{proposition}
\begin{proof} To verify condition (i) note that $\mathscr{D}(D)$ is dense in $H$, since $\text{Ran}(A)$ is. To see that $D$ is closed, assume that $u_n=Av_n+c_ny\in \mathscr{D}$ and  $v_n=Du_n$ converge to $u$ respectively to $v$. Then $(c_n)$ converges to $c\in \bC$, i.e. $u=Av+cy\in\mathscr{D}$ and $Du=v$. (ii) is obvious, as well as 
(iv) and (v). To see (iii), define the unbounded operator $B$ on $H$ by $\mathscr{D}(B)=\text{Ran}(A)$ and $BAu=u$.
Recall that $0\notin\sigma(A)$ and let $\{u_\la,~\la \in \sigma(A)\}$ be an orthonormal basis of eigenvectors of $A$. Then each $u_\la\in \mathscr{D}(B)$ and $Bu_\la=\la^{-1}u_\la$, which easily implies that $B$ is self-adjoint, and obviously, a restriction of $D$.\end{proof}	
	\end{section}

	\begin{section}{Preliminaries}\label{sec:prelim} We shall use standard notations for the usual function spaces on the upper half-plane $\bC_+$. Thus  $H^{2} = H^{2}(\mathbb{C}_{+})$ is the Hardy space and equals the (inverse) Fourier transform of $L^2(0,\infty)$. Equivalently $H^{2}(\mathbb{C}_{+})$ consists of  analytic functions $f$ in $\bC_+$ with 
		\begin{equation*}
		\norm{f}_{2}^{2} = \sup_{y > 0} \int_{-\infty}^{\infty} \lvert f(x+iy) \rvert^{2}dx < \infty.
		\end{equation*}
	$H^{\infty} = H^{\infty}(\mathbb{C}_{+})$ denotes the space of bounded analytic functions in $\mathbb{C}_{+}$.  $N(\mathbb{C}_{+})$ denote the functions of bounded type, or the Nevanlinna class in the upper-half plane, i.e. analytic functions of the form $f = g/h$, with $g, h \in H^{\infty}(\mathbb{C}_{+})$. All three spaces have identical definitions in the lower half-plane $\bC_-$.
    
	The mean-type  of  functions $h$ analytic in $\mathbb{C}_{\pm}$ is defined as
	\begin{equation*}
		\text{m.t.}(h) = \limsup_{y \to \pm\infty} \frac{\log \lvert h(iy) \rvert}{y}.
	\end{equation*}	
	For example, 	the mean-type of a function, $f$, of bounded type  is given by minus the exponent of $e^{i\alpha z}$ appearing in the inner-outer factorization of $f$.
	
	For entire functions $f$ one defines the (exponential) type	
	 type $\tau(f)$  as
		\begin{equation*}
		\tau(f) = \limsup_{\lvert z \rvert \to \infty} \frac{\log \lvert f(z) \rvert}{\lvert z \rvert}.
		\end{equation*}
			These notions are connected by a famous theorem of Krein 
	(see \cite{MR1400006}, p. 115, or \cite{MR1303780}, Part 2).	
		\begin{theorem}\label{Krein}
			Let $f$ be an entire function of bounded type in both half planes, $\mathbb{C}_{\pm}$. Then $f$ is of finite exponential type and its exponential type is equal to the maximum of its mean-type in the upper and lower half-planes.
		\end{theorem}    
	
We now turn to de Branges spaces of entire functions and start with Hermite-Biehler, briefly HB functions. 
These are  entire functions $E : \mathbb{C} \to \mathbb{C}$, such that $\lvert E(z) \rvert > \lvert E(\bar{z}) \rvert$, for all $z \in \mathbb{C}_{+}$.

	 To any HB function $E$ we associate the de Branges $\mathcal{H}(E)$ consisting of entire functions $f$ satisfying
		\begin{enumerate}[(i)]
			\item $f/E \in H^{2}(\mathbb{C}_{+})$,
			\item $f^{\#}/E \in H^{2}(\mathbb{C}_{+})$,
		\end{enumerate}
		where, as usual, for entire functions $f$ we denote $f^{\#}(z) = \overline{f(\overline{z})}$. The norm on $\mathcal{H}(E)$ is the one inherited from $L^{2}(\lvert E(t) \rvert^{-2}dt)$,
		\begin{equation*}
		\norm{f}_{\mathcal{H}(E)} = \norm{f/E}_{2} = \left( \int_{-\infty}^{\infty} \lvert f(t)E(t)^{-1} \rvert^{2}dt \right)^{1/2}.
		\end{equation*}
For the theory of de Branges spaces we refer to de Branges' book \cite{MR0229011}. Other good references for the subject include \cite{MR3890099, https://doi.org/10.48550/arxiv.1408.6022}. From property $(i)$ we see that $\mathcal{H}(E)$ is a reproducing kernel Hilbert space where the kernel is given by
		\begin{equation*}
		k_{\lambda}^{E}(z) = \frac{E(z)\overline{E(\lambda)} - E^{\#}(z)\overline{E^{\#}(\lambda)}}{2\pi i(\bar{\lambda} - z)}.
		\end{equation*}
	Of course, de Branges spaces are characterized by kernels of this form. The most classical  examples are the Paley-Wiener spaces which are the Fourier transforms of $L^2(-a,a)$. In this particular case, $E(z) = e^{-iaz}$, $a > 0$. 
		
	It is important to note that de Branges  spaces can also be introduced in an axiomatic way (see \cite{MR0229011}).	
		\begin{theorem}\label{deB_axiomatic}
			Let $\mathcal{H}$ be a Hilbert space consisting of of entire functions satisfying
			\begin{enumerate}[(i)]
				\item For $w \in \mathbb{C} $ the linear functional $f \mapsto f(w)$ is bounded.
				\item For $\lambda \in \mathbb{C} \setminus \mathbb{R}$, such that $f(\lambda) = 0$ we have $\frac{z-\bar{\lambda}}{z-\lambda}f \in \mathcal{H}$ and $\norm{\frac{z-\bar{\lambda}}{z-\lambda}f} = \norm{f}$.
				\item $f^{\#}(z) = \bar{f}(\bar{z}) \in \mathcal{H}$, and $\norm{f} = \norm{f^{\#}}$.
			\end{enumerate}
			Then there exists a HB function $E$, such that $\mathcal{H} = \mathcal{H}(E)$ as Hilbert spaces.
		\end{theorem}
		A subspace, $X$ of $\mathcal{H}(E)$ is called a \emph{de Branges subspace} if it is a de Branges space with respect to the norm inherited from  $\mathcal{H}(E)$. Equivalently, in view of the previous theorem, $X$  satisfies (i)-(iii) in that statement, hence it has the form $\cH(F)$ for some HB function $F$.
		 
		When $E$ has no real zeros the de Branges subspaces of $\mathcal{H}(E)$ have the following remarkable ordering property due to de Branges (see \cite{MR3890099, https://doi.org/10.48550/arxiv.1408.6022, MR0229011}).
		\begin{theorem}\label{ordering_thm}
			Suppose $E$ is a HB function. If $\mathcal{H}(F)$ and $\mathcal{H}(W)$ are two de Branges subspaces of $\mathcal{H}(E)$, and the corresponding HB functions $F$ and $W$ have no real zeros, then either $\mathcal{H}(F) \subset \mathcal{H}(W)$ or $\mathcal{H}(W) \subset \mathcal{H}(F)$.
		\end{theorem}
		For a HB function $E$ with no real zeros we define $\text{Chain}(\mathcal{H}(E))$ to be the collection of de Branges subspaces, whose HB functions have no real zeros. In general we call a collection of de Branges spaces a de Branges chain if it is totally ordered with respect to isometric inclusion.
		
		A de Branges space $\mathcal{H}(E)$ is called \emph{regular} (or short) if it is invariant under the backward shift operator, $Lf(z) = z^{-1}(f(z)-f(0))$. This is equivalent to the condition
		\begin{equation*}
		\frac{1}{(z+i)E(z)} \in H^{2}(\mathbb{C}_{+}).
		\end{equation*}
		By extension we call $E$ a regular HB function.
		
		An important property of such spaces is the following result
	(\cite{https://doi.org/10.48550/arxiv.1408.6022},  Theorem 9,  see also \cite{MR3890099, MR0229011}).
		\begin{theorem}
			Let $\mathcal{H}(E)$ be a regular de Branges space. If $\mathcal{H}(F)$ is a de Branges subspace and $F$ has no real zeros, then $\mathcal{H}(F)$ is regular.
		\end{theorem}
	It follows  that if $E$ is a regular HB function then it is of finite exponential type $\tau(E)$ which  equals its mean-type in $\mathbb{C}_{+}$ defined above.

	Regular HB functions are related  to canonical systems in the following way.
			 Let $\bH$ be a Hamiltonian as in \S \ref{sec:canonical_systems} and $\Theta = (\Theta_{+}, \Theta_{-})^{T}$ be the unique solution to the canonical system
        \begin{equation*}
            \Omega \partial_{t} X(t) = z\bH(t)X(t), \; t \in (0,\ell),
        \end{equation*}
        satisfying $\Theta(0) = (1,0)^{T}$. Then $E_{t}(z) = \Theta_{+}(t,z) + i \Theta_{-}(t,z)$ is a regular HB function for each $t \in (0,\ell]$, except in the degenerate case
        \begin{equation*}
            \bH(s) = \begin{pmatrix}
                0 & 0\\ 0 & 1
            \end{pmatrix}, \text{ for almost every } s \in (0,t).
        \end{equation*}
        In addition,
        \begin{equation*}
            \text{Chain}(\mathcal{H}(E_{t})) = \left\{ \mathcal{H}(E_{s}) : s \in (0, t], s \text{ is regular}\right\},
        \end{equation*}
        Conversely, it is a deep theorem due to de Branges that to every regular de Branges space $\mathcal{H}(E)$ with $E(0) = 1$ there exists a (essentially unique) Hamiltonian $\bH$ with constant trace equal to $1$ and an interval $(0,t]$, such that $E = E_{t}$ and
        \begin{equation*}
            \text{Chain}(\mathcal{H}(E)) = \left\{ \mathcal{H}(E_{s}) : s \in (0, t], s \text{ is regular} \right\}.
        \end{equation*}

		For each regular HB function $E$ with no real zeros, $E(0) = 1$, the type of $E$ may also be calculated from the associated canonical system using the following Krein-de Branges formula
		\begin{equation}\label{eq:krein_de_branges_formula_prelim}
		\tau(E) = \int_{0}^{\ell} \sqrt{\det{\bH(x)}} dx,
		\end{equation}
		see \cite{MR3890099, MR133456, https://doi.org/10.48550/arxiv.1408.6022, MR0229011}.
		\begin{remark}\label{remark:type}
			Any $f \in \mathcal{H}(E)$ satisfies $\tau(f) \leq \tau(E)$. In particular, if the exponential type of $E$ is zero then the exponential type of all elements in $\mathcal{H}(E)$ is zero.
		\end{remark}
	
This connection to canonical systems can be used to establish the existence of de Branges subspaces of a given type. We include a proof of the result for the sake of completeness. A variant can be found in de Branges' book \cite{MR0229011} as well.
		\begin{lemma}\label{lemma:subspace_of_given_type}
			Let $E$ be a regular HB function with $\tau(E) > 0$. Then for each $0 < h < \tau(E)$ there exists a HB function $F$ zero free on $\mathbb{R}$ with $\tau(F) = h$ and $\mathcal{H}(F)$ is a de Branges subspace of $\mathcal{H}(E)$.
			\begin{proof}
				We may suppose $E(0) = 1$. Let $\bH$ be a canonical system associated with $E$. Then
				\begin{equation*}
				\tau(E) = \int_{0}^{\ell} \sqrt{\det{\bH(x)}} dx.
				\end{equation*}
				The right-hand side is continuous in $\ell$ and therefore it is possible to find $\delta > 0$, such that
				\begin{equation*}
				\tau(E)-\epsilon = \int_{0}^{\ell-\delta} \sqrt{\det{\bH(x)}} dx.
				\end{equation*}
				It is also possible to choose $\delta$, such that $\ell-\delta$ does not belong to a singular interval, since $\det(\bH(x)) = 0$, almost everywhere on singular intervals. Then the canonical system $\bH$ on $(0,\ell-\delta)$ gives a de Branges subspace, $\mathcal{H}(E_{\epsilon})$ of $\mathcal{H}(E)$ with $\tau(E_{\epsilon}) = \tau(E) - \epsilon$.
			\end{proof}
		\end{lemma}

Finally, we discuss the concept of \emph{associated functions} to a given de Branges space.
	\begin{definition}\label{def:associated}
		An entire function $S$ is said to be associated with a de Branges spaces $\mathcal{H}(E)$ if whenever $S(\lambda) \neq 0$ and $f \in \mathcal{H}(E)$, we have
		\begin{equation*}
			\frac{f(z)S(\lambda)-f(\lambda)S(z)}{z-\lambda} \in \mathcal{H}(E).
		\end{equation*}
	\end{definition}
	Examples of associated functions include $E$, $A = 2^{-1}(E + E^{\#})$, and $B = -i2^{-1}(E - E^{\#})$. An entire function $S$ is associate with the de Branges space $\mathcal{H}(E)$ if and only if $S/E$, $S^{\#}/E$ are of bounded type, non-positive mean-type in $\mathbb{C}_{+}$ and
	\begin{equation*}
		\int_{-\infty}^{\infty} \frac{1}{1+t^{2}} \frac{\lvert S(t) \rvert^{2}}{\lvert E(t) \rvert^{2}}dt < \infty.
	\end{equation*}
	This characterization can be reformulated as follows. 
	\begin{proposition}\label{exp-associated} The entire function $S$ is associated to $\cH(E)$ if and only if $S\in \cH((z+i)E)$. In particular, if  $E$ is regular then for $\beta\in \mathbb{R}$, the exponential function $z\mapsto e^{i\beta z}$ is associated to $\cH(E)$  if and only if $|\beta|\le\tau(E)$. 
	\end{proposition}

	\end{section}

	\begin{section}{The model}\label{sec:model}
		We now turn to the construction of the analytic model for an admissible pair $D$ and $V$. The idea is fairly straightforward and reminiscent from the  spectral theory of linear differential operators. It is based on the  simple observation given in the lemma below. The more subtle part is to identify the objects resulting from the constructions.

		Throughout in what follows we shall denote by $A$ the densely defined self-adjoint restriction of $D$ given in assumption (iii) and by $\Lambda=\sigma(A)$. Moreover, let $\phi_0$ a non-zero vector in $\ker D$. Recall that by assumption (ii) $\phi_0$ spans $\ker (D)$.
		\begin{lemma}\label{one-dim ker} For each $\lambda\in \bC,~\ker(D-\la I)$ is spanned by $$\phi_\la=(I-\la V)^{-1}\phi_0.$$ Moreover, the set
			$$\{u_\la=\frac1{\|\phi_\la\|}\phi_\la:~\la \in \Lambda\},$$
			is an orthonormal basis of $H$.
			
		\end{lemma}
		
		\begin{proof} The statement is almost self-explanatory. We have that 
			$$\phi_\la=(I-\la V)^{-1}\phi_0=\phi_0+\la V
			(I-\la V)^{-1}\phi_0=\phi_0+\la V\phi_\la,$$ so that $\phi_\la \in \mathscr{D}(D)$ and $D\phi_\la=\la DV\phi_la=\la \phi_\la$, i.e. $\phi_\la\in \ker(D-\la I)$. Conversely, if $Du=\la u$, it follows that $D(I-\la V)u=0$, i.e. $u$ is a multiple of $\phi_\la$. The second part follows from the first, the fact that the resolvent of $A$ is compact, and the spectral theorem.
			
		\end{proof}
		With the lemma in hand we can proceed to construct our model.
        
        Let $\mathcal{H}$ be the space of entire functions
		\begin{equation}\label{model_space}
		\mathcal{H} = \left\{ \hat{f} \in \mathcal{O}(\mathbb{C}) : \hat{f}(\lambda) = \langle f, \phi_{\bar{\lambda}} \rangle_{H} \text{, for some } f \in H \right\}.
		\end{equation}
		The second part of  Lemma \ref{one-dim ker} shows that $\hat{f}$ is uniquely determined by $f$, as well as by  the values $\hat{f}(\la),~\la\in \Lambda$. In particular, we can define a norm on $\mathcal{H}$ by
		\begin{equation}\label{model_space_norm}
		\|\hat{f}\|_\mathcal{H}=\|f\|_H,\end{equation}
		which makes the map $\mathcal{W} : H \to \cH,~\cW f=\hat{f}$ unitary.

		The unitary map $\mathcal{W}$ can be seen as a generalized Fourier transform associated to the operator $D$. If $D = -id/dx$, then $\mathcal{W}$ is the classical (inverse) Fourier transform and if $D$ is a regular Schrödinger operator it is nothing but the classical Weyl-Titmarsch transform. 
		The usual gain from considering such analytic models is a much more tractable form of the operators involved in the process. 
		
		We shall relate the operators $D,V$ and  the self-adjoint restriction $A$ of $D$ acting on $H$, to the forward and backward shifts $M_z,L$ acting on $\cH$, More precisely, we consider the operator $M_z$  of multiplication by the independent variable acting on $\cH$ ($M_zf=zf$), with the maximal domain of definition
		$$\mD(M_z)=\{f\in \cH:~zf\in \cH\},$$ 
		and the backward shift $L$ on $\cH$ formally defined by 
		\begin{equation}\label{backward}
		Lg(z)=\frac{g(z)-g(0)}{z},\quad g\in \cH,\,z\in \bC.	
		\end{equation}
		\begin{proposition}\label{basic_op} With the above notations we have:\\
			(i) $\cW V^*\cW^{-1}=L$,\\	
			(ii) If  $x\in \mD(D^*)$ then $\cW x\in \mD(M_z)$ and $\cW D^*x=M_zWx$.\\
			(iii) $M_z$ is a closed symmetric operator with $$(M_z-\la I)\mD(M_z)=\{f\in \cH:~f(\la)=0\}, \quad \la\in \bC,$$
			
		\end{proposition}
		
		\begin{proof} 
			(i) If $x\in H, ~\la \in \bC$,
			\begin{align*}\cW V^*x(\la)&=\langle x, V\phi_{\overline{\la}} \rangle =\langle x, V(1-\overline{\la} V)^{-1}\phi_0 \rangle\\&
			=\frac1{\la}(\langle x, (1-\overline{\la} V)^{-1}\phi_0 \rangle -\langle x,\phi_0\rangle)=L\cW x(\la).		
			\end{align*}
			(ii)  If $x\in \mD(D^*)$ then
			$$\cW D^*x(\la)=\langle x,D\phi_{\overline{\la}}\rangle=
			\la \cW x(\la),\quad \la\in \bC,$$ which gives the assertion.\\
			(iii) If $(f_n)$ is a sequence in $\mD(M_z)$ with $f_n\to f,~M_zf_n\to g$ in $\cH$ we use the continuity of point evaluations on $\cH$ to obtain
			$$g(\la)=\lim_{n\to\infty}\la f_n(\la)=\la\lim_{n\to\infty} f_n(\la)=\la f(\la),$$
			which  shows that the graph of $M_z$ is closed. To see that $M_z$ is symmetric we use the following claim:
			
			{\it If $x,y\in H$ satisfy $\langle x,\phi_0\rangle=\langle y,\phi_0\rangle=0$ then $\langle V^*x,y\rangle=\langle x, V^*y\rangle$.} By (i), an equivalent formulation is: If $u,v\in \cH$ satisfy $u(0)=v(0)=0$ then \begin{equation}\label{M_symm}\langle Lu,v\rangle=\langle u, Lv\rangle \end{equation}
			With \eqref{M_symm} in hand, for $f,g\in \mD(M_z)$ write $f=Lu,~g=Lv$ with $u=M_zf,~v=M_zg$ and since $u(0)=v(0)=0$, an application of \eqref{M_symm} gives the symmetry of $M_z$.	
			
			To verify the claim, recall that $H$ has an orthonormal basis of eigenvectors of $A$ which have the  form
			$$e_\la=c_\la(I-\la V)^{-1}\phi_0,\quad \la\in \Lambda\subset \mathbb{R}.$$
			We have 
			$$(I-\la V)^{-1}\phi_0=\phi_0+\la V(I-\la V)^{-1}\phi_0,$$
			and from $\la\in \mathbb{R}$,  $\langle x,\phi_0\rangle=\langle y,\phi_0\rangle=0$ we obtain 
			$$
			\langle V^*x,y\rangle =\sum_{\la\in \Lambda}\langle x,Ve_\la\rangle\langle e_\la,y\rangle = \sum_{\la\in \Lambda\setminus\{0\}}\frac1{\la}\langle x,e_\la\rangle\la\langle V e_\la,y\rangle =\langle x,V^*y\rangle$$
			which gives the claim and concludes the proof.
		\end{proof}
		
		Another piece of information of crucial 
		importance  regards the structure of the underlying model space $\cH$.
		This is described in the main result of this section. Recall that the HB function $E$ is called regular if $[(z+i)E]^{-1}\in H^2(\bC_+)$.
		
		\begin{theorem}\label{theorem:space}
			There exists a regular HB function, $E$, with no real zeros, and a constant $\alpha \in \mathbb{R}$, such that $e^{i\alpha z}\mathcal{W}H = \mathcal{H}(E)$ as Hilbert spaces.
		\end{theorem}
		
		The result can be seen as a theorem related to the axiomatic definition of de Branges spaces of entire functions described in Theorem \ref{deB_axiomatic}. More precisely, Theorem \ref{theorem:space} can be deduced from the description of Hilbert spaces of entire functions which only satisfy the first two conditions characterizing de Branges spaces but not necessarily the third. 
		
		\begin{theorem}\label{theorem:near_invariance}
			Let $\mathcal{N}$ be a Hilbert space of entire functions without common zeros and satisfying
			\begin{enumerate}[(i)]
				\item For $\la \in \mathbb{C}$ the linear functional $f \mapsto f(\la)$ is bounded,
				\item For $\lambda \in \mathbb{C} \setminus \mathbb{R}$, such that $f(\lambda) = 0$ we have $\frac{z-\bar{\lambda}}{z-\lambda}f \in \mathcal{N}$ and $\norm{\frac{z-\bar{\lambda}}{z-\lambda}f} = \norm{f}$.
			\end{enumerate}
			Then there exists a zero free entire function $u$, satisfying $uu^{\#}\equiv 1$ and a HB function $E$ with no real zeros, such that
			\begin{equation*}
			u\mathcal{N} = \mathcal{H}(E),
			\end{equation*}
			and multiplication by $u$ is a unitary map from 
			$\mathcal{N}$ onto $\mathcal{H}(E)$. If, in addition, one of the following conditions holds
			\begin{enumerate}
				\item $\mathcal{N}$ is a subspace of a de Branges space $\mathcal{H}(W)$,
				\item $\mathcal{N}$ is invariant for the backward shift $L$ and $\sigma(L|\mathcal{N})=\{0\}$, 
			\end{enumerate}
			then $u(z) = e^{i\alpha z}$, for some $\alpha \in \mathbb{R}$ and all $z\in \bC$. Moreover, if the second assumption holds then $E$ is regular and $|\alpha|\le\tau(E)$.
		\end{theorem}
		\begin{proof} By assumption (i), $\mathcal{N}$ has a reproducing kernel $k^\mathcal{N}$.
			By Theorem 4.1 in \cite{MR3004956}, or by  the proof of Theorem 23 in \cite{MR0229011} it follows that $k^{\mathcal{N}}$ has the form
			\begin{equation}\label{kernel_form}
			k^{\mathcal{N}}(z,\la)=	k^{\mathcal{N}}_{\lambda}(z) = \frac{F(z)\overline{F(\lambda)} - G(z)\overline{G(\lambda)}}{2\pi i(\bar{\lambda} - z)},
			\end{equation}
			where $F$ and $G$ are the entire functions given by
			\begin{equation}\label{kernel_structure}
			\begin{cases}
			F(z) = \sqrt{\pi}k_{i}^{\mathcal{N}}(z)k_{i}^{\mathcal{N}}(i)^{-1/2}(z+i) \\
			G(z) = \sqrt{\pi}k_{-i}^{\mathcal{N}}(z)k_{-i}^{\mathcal{N}}(-i)^{-1/2}(z-i)
			\end{cases}.
			\end{equation}
			Note that \eqref{kernel_form} together with the assumption that  $\mathcal{N}$ has no common zeros implies that $F$ and $G$ have no common zeros in $\bC$. Indeed, by \eqref{kernel_form}
			$F(\la)=G(\la)=0$ implies $k^{\mathcal{N}}_{\lambda} \equiv 0$ which contradicts this assumption on $\mathcal{N}$.\\
			Now set $\la=\overline{z}\in \bC$ in \eqref{kernel_form} and note that since the denominator in \eqref{kernel_form} vanishes, the numerator must vanish as well, i.e. 
			\begin{equation}\label{F_G_sharp}
			F(\bar{z})\overline{F(z)} = G(\bar{z})\overline{G(z)}\Leftrightarrow F^\#(z)F(z)=G^\#(z)G(z),\quad z\in \bC.
			\end{equation}
			In particular, neither $F$ nor $G$ can have a real zero $x$, otherwise from $|F(x)|^2=|G(x)|^2$ we  would conclude that $x$ is a common zero of these functions which cannot hold by the above argument. Moreover, from above if $F(\la)=0$, then $G(\la)\ne0$, hence by \eqref{F_G_sharp} it follows that $G^\#(\la)=0$  and the order of the zero $\la$ for $G^\#$ is at least as large at the order of  the zero $\la$ for $F$. Thus the function $v=G^\#/F$ is entire and similarly, $F^\#/G=1/v^\#$ is entire, i.e. $v$ has no zeros in $\bC$.
			Finally, from the equality $|F(x)|^2=|G(x)|^2,~x\in \mathbb{R}$ we see that $|v|=1$ on $\mathbb{R}$. Equivalently, $vv^\#=1$. 
			
			Set $u=v^{1/2}$ and $E=uF$. Then $E$ is zero-free on $\mathbb{R}$,  $E^\#=u^\#F^\#=\frac{F^\#}{u}$ and the inequality $|E(z)|>|E^\#(z)|, z\in \bC_+$ is equivalent to $$|u(z)|^2|F(z)|=|G^\#(z)|>|F^\#(z)|, \quad 
			z\in \bC_+,$$
			which follows from 
			$$0<|k_\la^\mathcal{N}(\la)|=\frac{|F(\la)|^2-|G(\la)|^2}{4\text{Im }\la},\quad \la\in \bC_-.$$
			Thus, $E$ is a HB function	and the equalities 
			$$u(z)\overline{u(\la)}k_\la^\mathcal{N}(z)=k_\la^E(z),\quad k_\la^\mathcal{N}(z)=\frac1{u(z)\overline{u(\la)}}k_\la^E(z),$$
			give that multiplication by $u$ is unitary.
			
			To prove the second part, we shall show that both conditions imply that $u$ is of bounded type in $\bC_+$. Since $|u|=1$ on the real line, it will  follow that $u(z)=e^{i\alpha z}$, for some real number $\alpha$ and all $z\in \bC$.\\
			1.  If $\mathcal{N}$ is a closed subspace of $\cH(W)$ then from the formulas defining $F,G$ in \eqref{kernel_structure} it follows that 
			\begin{equation*}
			u^2(z) = \frac{(k^{\mathcal{N}}_{-i})^{\#}(z)(k^{\mathcal{N}}_{-i}(-i))^{-1/2}}{k^{\mathcal{N}}_{i}(z)(k^{\mathcal{N}}_{i}(i))^{-1/2}}.
			\end{equation*}
			Since $k^{\mathcal{N}}_{i}$ and $k^{\mathcal{N}}_{-i}$ belong to $\mathcal{H}(W)$, it follows that $k^{\mathcal{N}}_{i}/W, (k^{\mathcal{N}}_{-i})^{\#}/W \in H^{2}(\mathbb{C}_{+})$ and so as a quotient of two $H^2-$functions, $u^2$ is  of bounded type and obviously so is $u$.\\
			2.  Assume that $\mathcal{N}$ is backward-shift invariant denote this operator by  $L$ and assume that $\sigma(L)=\{0\}$. A simple computation reveals that if $\lambda\in \bC$, the operator $L(1-\lambda L)^{-1}$ is given by 
			$$L(I-\la L)^{-1} h(z)=\frac{h(z)-h(\la)}{z-\la},\quad z\in \bC.$$
			If $u\mathcal{N}=\cH(E)$, $\la\in \bC$ and $f\in \cH(E)$, then from $uL(I-\la L)^{-1} u^\#f\in \cH(E)$ we obtain after a direct computation using $uu^\#=1$ that the function  $$z\to\frac{f(z)u(\la)-f(\lambda)u(z)}{z-\lambda}$$ belongs to  $\cH(E)$, that is,  $u$ is  associated to $\cH(E)$. By the definition of this space, it follows that   $u^\#$ is  associated to $\cH(E)$ as well. Thus $u/E,~u^\#/E\in (z+i)H^2(\bC_+)$  and since $uu^\#=1$, we have that 
			$$E^{-2}=\frac{u}{E}\frac{u^\#}{E}\in (z+i)^2H^1(\bC_+).$$
			Then $E$ is regular, in particular, of bounded type and consequently $u=E(u/E)$ is of bounded type in $\bC_+$.  Then $u(z)=e^{i\alpha z}$ and $|\alpha| \le \tau(E)$ because $u$ is associated to $\cH(E)$.
		\end{proof}

		\begin{proof}[Proof of Theorem \ref{theorem:space}]
			By Proposition \ref{basic_op} (i) the space $\mathcal{H} = \mathcal{W}H$ is $L-$invariant with $\sigma(L|\mathcal{H})=\sigma(V^*|H)=\{0\}$, hence by the previous theorem  it will suffice to verify that $\mathcal{H}$ satisfies the  statements:
			\begin{enumerate}[(i)]
				\item For $\la \in \mathbb{C}$ the linear functional $\la \mapsto f(\la)$ is bounded,
				\item For $f\in \cH$ and $\lambda \in \mathbb{C} \setminus \mathbb{R}$, such that $f(\lambda) = 0$ we have $\frac{z-\bar{\lambda}}{z-\lambda}f \in \mathcal{H}$ and $\norm{\frac{z-\bar{\lambda}}{z-\lambda}f} = \norm{f}$.
			\end{enumerate}
			(i) is obvious by  definition, since for $f=\cW x$
			$$|f(\la)|=|\langle x, \phi_{\overline{\la}}\rangle|\le \|x\|_H\|\phi_{\overline{\la}}\|_H=\|f\|_{\cH}\|\phi_{\overline{\la}}\|_H.$$
			For part (ii) the equality     \begin{equation*}
			L(I-\lambda L)^{-1}f(z) = \frac{f(z)-f(\lambda)}{z-\lambda},
			\end{equation*}
			and the identity
			\begin{equation*}
			\frac{z-\bar{\lambda}}{z-\lambda} = 1 + \frac{\lambda}{z-\lambda} - \frac{\bar{\lambda}}{z-\lambda},
			\end{equation*}
			show that $\frac{z-\bar{\lambda}}{z-\lambda}f \in \mathcal{H}$, when $f\in \cH$ and $\lambda \in \mathbb{C} \setminus \mathbb{R}$, such that $f(\lambda) = 0$. To verify the equality of norms, we use Proposition \ref{basic_op} (iii) and note that the equality is equivalent to the fact  that the Cayley transform of the symmetric operator $M_z$ is isometric.
		\end{proof}
		
		Let us briefly discuss the applications of Theorem \ref{theorem:space} to some of the situations considered in Section \ref{sec:examples}.
		
		\begin{examples}\label{ex_model} 1) {\rm \emph{Standard differentiation  and Volterra.}  Here   $H=L^2(a,b)$, $D = -i\frac{d}{dx}$ defined on the absolutely continuous functions in $[a,b]$ and 
				\begin{equation*}
				Vf(x) = i\int_{a}^{x} f(t)dt.
				\end{equation*}
				Since $\phi_z(t)=e^{i(t-a)z}$ a direct computation yields
				$$\cW f(z)=e^{-iaz}\int_a^bf(t)e^{itz}dt,\quad \cW H=e^{i\frac{b-a}{2}z}PW(-\frac{b-a}{2}, \frac{b-a}{2})=e^{i\frac{b-a}{2}z}\cH(E),$$ with $E(z)=e^{-i\frac{b-a}{2}z}$.\\
				2) \emph{ Regular Schr\"odinger operators.} In this case $D$ and $V$ are defined as in \S \ref{Schr_reg} The interesting fact here is that $\cW H$ is a de Branges space $\cH(E)$, in particular,  in the notation in Theorem \ref{theorem:space} we have $\alpha=0$.  Indeed, it follows easily that in this case we have \begin{equation}\label{sharp1}\phi_{\overline{z}}=\overline{\phi_z},\end{equation}
				and since $H=L^2(a,b)$ we obtain 
				\begin{equation}\label{sharp2}\cW^{-1}(\cW f)^\# =\overline{f},\quad f\in L^2(a,b),\end{equation}
				i.e. $g\to g^\#$ is an isometry on $\cW H$. Then the claim follows by the  above proof of Theorem \ref{theorem:space} and  Theorem \ref{deB_axiomatic}.\\   
				3) \emph{Canonical systems}. Here $D,V$ are described in \S \ref{sec:canonical_systems} with help of the Hamiltonian $\bH$. The same argument as above shows that 
				\eqref{sharp1} and \eqref{sharp2} continue to hold with obvious modifications, so that $\cW H=\cH(E)$ and $\alpha=0$ in Theorem \ref{theorem:space}.  
			}
		\end{examples}

		Even if assumption 1. in Theorem \ref{theorem:near_invariance} is not directly involved in the description of our model space $\cH=\cW H$, the condition is interesting in its own right and the corresponding conclusion has  useful applications related to the unicellularity of generalized Volterra operators as well as to the study of invariant subspaces of $D$.
		
		\begin{corollary}\label{subspace:deBranges}
			Let $\mathcal{N}$ be a Hilbert space of entire functions satisfying the assumptions (i) and (ii) in  Theorem \ref{theorem:near_invariance}. Then   $\mathcal{N}\subset	e^{i\delta z}\cH(G)$ for some $\delta\in \mathbb{R}$ and for some regular HB function $G$ if and only if $\cN=e^{i\gamma z}\cH(F)$ for some
			$\mathcal{H}(F)\in$ Chain$(\mathcal{H}(G))$ and $\gamma\in \mathbb{R}$ with 
			$$|\gamma-\delta|\le \tau(G)-\tau(F).$$	
			The de Branges space $\mathcal{H}(F)$ and $\gamma\in \mathbb{R}$ are uniquely determined by $\cN$.
		\end{corollary}
	\begin{proof}
	If $\mathcal{N}\subset	e^{i\delta z}\cH(G)$ we apply 
	Theorem \ref{theorem:near_invariance} to $e^{-i\delta z}\cN$ to conclude that  $e^{-i\delta z}\cN=e^{i(\gamma-\delta) z}\cH(F)$,
	with $\gamma\in \mathbb{R}$ and $F$ a HB function. If 
	$f \in \mathcal{H}(F)$ we have $e^{i(\gamma-\delta) z}f\in \cH(G)$ and also
	$(e^{i(\gamma-\delta)z}f^{\#})^{\#}=e^{-i(\gamma-\delta)z}f\in \cH(G)$. This easily implies that $f\in \cH(G)$, that is, 
	$\mathcal{H}(F)\in$ Chain$(\mathcal{H}(G))$,  since $G$ is regular.  In this case $F\in \cH(z+i)G)$ and if $f\in \cH(F)$, from 
	$$e^{i(\gamma-\delta) z} \frac{f}{G}=e^{i(\gamma-\delta) z} \frac{f}{F}\frac{F}{G}$$
	we see that the functions $e^{i(\gamma-\delta) z} \frac{f}{G},~e^{i(-\gamma+\delta) z} \frac{f^\#}{G}$ belong to $H^2(\bC_+)$ if and only if the inequality in the statement holds. The converse is straightforward.	 Finally, for the uniqueness, use the above proof to conclude that $e^{i\gamma z}\cH(F)=e^{i\delta z}\cH(G)$ implies $\cH(F)=\cH(G)$, hence $\gamma=\delta$ from the inequality in the statement	
	\end{proof}

	\begin{remark}{\rm The first part of the above proof shows that the general form $\cN=e^{i\gamma z}\cH(F)$ of nearly invariant subspaces $\cN$ of $e^{i\delta z}\cH(G)$, without common zeros,
	remains true without the assumption that $G$ is a regular HB  function.}
	\end{remark}
	
		\end{section}
	
	\begin{section}{Generalized Volterra operators}\label{sec:new_deBranges_results}
		The aim of this section is to explore invariant subspaces of right inverses of the generalized differentiation operator $D$, especially to characterize those which are unicellular.		
		Let us begin with the observation that the construction of our model space $\cH$ depends essentially on the choice of the right inverse $V$ of the unbounded operator $D$ considered in this paper. Of  course, $D$ has a collection of right inverses which are all rank-one perturbations of the chosen $V$. They have the form
		\begin{equation}
		\label{right-inverses} V_x=V+\langle \cdot, x\rangle \phi_0,
		\end{equation}
		where $x\in H$ is arbitrary but fixed. All of these operators are compact, but they are not necessarily quasi-nilpotent, and not necessarily unicellular. For example, if $D$ is the classical differentiation operator with maximal domain in $L^2(a,b)$ and $$Vf(t)=\int_a^tf(s)ds,\quad f\in L^2(a,b),$$
		then $\sigma(V_x)=\{0\}$ if and only if $$V_xf(t)=\int_x^tf(s)ds,\quad f\in L^2(a,b),$$
		for some $c\in [a,b]$ and $V_x$ is unicellular if and only if $x=a$, or $x=b$.
		It turns out that an analogous result continues to hold in the very general context as well.  For  this purpose it will be more convenient to represent $V_x$ as operators acting on the model space. Since $$V_x^*=V^*+\langle \cdot, \phi_0\rangle x,$$
		we obtain from Proposition \ref{basic_op} (i)
		\begin{equation}\label{Vx}\cW V_x^*\cW^{-1}f(z)=\frac{f(z)-f(0)}{z}+f(0)\cW x(z)
		=\frac{f(z)-f(0)(1-z\cW x(z))}{z}.
		\end{equation}
		To establish the analogy to the standard example described above, we start by describing the quasi-nilpotent right inverses of $D$. 
		
		\begin{proposition}\label{quasi-nilpotent-ri}	
			The compact operator $V_x$ has spectrum $\{0\}$ if and only if there exists $\beta\in \mathbb{R}$ with $\lvert \beta +\alpha\rvert \leq \tau(E)$, such that $x=x_\beta$ with		
			$$\cW x_\beta(z)=\frac{1-e^{i\beta z}}{z}.$$
		\end{proposition}	
		\begin{proof} Since $V_x$ is compact, it is quasi-nilpotent if and only if $V_x^*$ has no nonzero eigenvalues. A straightforward computation based on \eqref{Vx} shows that  an eigenvector $y\in H$ of $V_x^*$ corresponding to the eigenvalue $\la\in \bC$ satisfies
			$$f(z)=\cW y(z)=f(0)\frac{1-z\cW x(z)}{1-\la z},\quad z\in \bC.$$
			Thus, $\sigma(V_x^*)=\{0\}$ if and only if the entire function $1-z\cW x(z)$ has no zeros in $\bC$. Using Theorem \ref{theorem:space} together with Theorem \ref{Krein} we conclude that the function $g(z)=1-z\cW x(z)$  is of exponential type, hence
			$$g(z)=e^{az},\quad z\in \bC,$$
			for some fixed constant $a\in \bC$. Since $z\to e^{i\alpha z} \cW V_x^*\cW^{-1}f(z), z\to e^{i\alpha z}f(z)$ belong to $\cH(E)$ it follows easily that $$z\to \frac{e^{i\alpha z}g(z)}{E(z)}=\frac{e^{(a+i\alpha)z}}{E(z)} \in (z+i)H^2(\bC_+).\,\, z\to \frac{e^{-i\alpha z}g^\#(z)}{E(z)}=\frac{e^{(\bar{a}-i\alpha)z}}{E(z)} \in (z+i)H^2(\bC_+).$$ Since $E$ is regular, these can hold if and only if   $a=i\beta,~\beta\in \mathbb{R}$ and $\lvert \beta +\alpha\rvert \leq \tau(E)$, where the last inequality follows by comparing the mean-types of the functions involved.
		\end{proof}	
 For simplicity we shall denote the quasi-nilpotent left inverses of $D$ described in Proposition \ref{quasi-nilpotent-ri} by $V_\beta,~\beta\in \mathbb{R},~|\beta+\alpha| \le \tau(E)$. Let us now turn to the unicellularity of these operators.  Our main tool is Corollary \ref{subspace:deBranges} which immediately provides a complete characterization of the invariant subspaces of $\cW V_\beta^*\cW^{-1}$. 
 
 \begin{proof}[Proof of Theorem \ref{thm:1'}]  $V_\beta$ is unicellular if and only if $L_\beta=\cW V_\beta^*\cW^{-1}$ is 
 unicellular. Since $L_\beta$ is quasi-nilpotent, every nonzero closed invariant subspace $\cN$ is also invariant for $(I-\la L_\beta)^{-1}L_\beta$ given by $$(I-\la L_\beta)^{-1}L_\beta f(z)=\frac{f(z)-f(\la)e^{i\beta(z-\la)}}{z-\la},\quad z,\la\in \bC,$$
 which shows that such a subspace $\cN$ is nearly invariant without common zeros. According to Corollary \ref{subspace:deBranges}
 $$\cN=e^{i\gamma z}\cH(F),\quad \gamma\in \mathbb{R},\, \cH(F) \in \text{Chain}(\cH(E)).$$
 These subspaces are characterized by the inequalities  in  Corollary \ref{subspace:deBranges} and in Proposition \ref{exp-associated}, since $L_\beta-$invariance is equivalent to the fact that $e^{i(\beta-\gamma)z}$ is associated to $\cH(F)$. The inequalities read
 \begin{equation}\label{v_beta_star_inv}
 |\gamma+\alpha|\le \tau(E)-\tau(F),\quad 	|\beta-\gamma|\le \tau(F).
 \end{equation} 
 For example, if $\beta=\beta_{+} = \tau(E) - \alpha$,  we see that $|\gamma+\alpha|\le \tau(E)$,  and the second inequality gives  $\gamma+\alpha\ge \tau(E)-\tau(F)$, that is, \eqref{v_beta_star_inv} is equivalent to 
 $$\gamma+\alpha=\tau(E)-\tau(F).$$ Thus if $\cN=e^{i\gamma z}\cH(F),~\cK=e^{i\delta z}\cH(G)$ are $L_{\beta_+}-$invariant with $\cH(F), \cH(G) \in \text{Chain}(\cH(E))$, say  $\mathcal{H}(F)\subset \mathcal{H}(G)$, then
 $$|\gamma-\delta|=\tau(G)-\tau(F),$$
 and another application of  Corollary \ref{subspace:deBranges} gives $\cN\subset\cK$, that is, $L_{\beta_+}$ is unicellular. Similarly, it follows that 
 $L_{\beta_-}$ is unicellular.
 
 If $\beta_-<\beta_+$, note first that $\tau(E)$ must be positive. Now choose $\beta\in (\beta_-,\beta_+)$ and use Lemma \ref{lemma:subspace_of_given_type}  to find $\cH(F),\cH(G)\in \text{Chain}(\cH(E))$ such that $$\tau(F)=\frac{\beta-\beta_-}{2},\quad \tau(G)=\frac{\beta_+-\beta}{2},$$
 and choose $\delta=\frac{\beta_++\beta}{2},~\gamma=\frac{\beta_-+\beta}{2}$. Then from $$\beta-\gamma=\gamma-\beta_-=\frac{\beta-\beta_-}{2}=\tau(F),\quad
 \delta-\beta=\beta_+-\delta=\frac{\beta_+-\beta}{2}=\tau(G),$$
 we deduce
 $$\beta-\gamma=\tau(F),\,|\gamma+\alpha|=\tau(E)-\tau(F),\quad \delta-\beta=\tau(G),\,|\delta+\alpha|=\tau(G),$$
 which shows that both spaces $\cN=e^{i\gamma z}\cH(F),~\cK=e^{i\delta z}\cH(G)$ are contained in $\cW H$ and invariant for $V_\beta$. On the other hand, 
 $$\delta-\gamma=\tau(F)+\tau(G)>|\tau(F)-\tau(G)|,$$ hence, again 
 by  Corollary \ref{subspace:deBranges} we see that  
 neither subspace is contained in the other subspace, that is $V_\beta$ is not unicellular. 
		\end{proof}
	\begin{corollary}\label{thm:1}
		Let $D$ and $V$ be an admissible pair. Then $V$ is unicellular if and only if $\tau(E) = \lvert \alpha\rvert$.
	\end{corollary}
A direct application of the proof gives the general form of $V_\beta-$invariant subspaces already mentioned in the Introduction.
\begin{corollary}\label{v_beta_inv} If $\{0\}\ne \mathcal{M} \subset H$ is a $V_\beta-$invariant subspace, $|\beta+\alpha\|\le\tau(E)$, then there exist $\gamma\in \mathbb{R}$ and a HB function $F$ with $\cH(F)\in \text{Chain}(\cH(E))$ such that 	
	$$\cW\mathcal{M}=(e^{i\gamma z}\cH(F))^\perp,\quad \gamma\in \mathbb{R}.$$
	 Moreover, we have	\begin{equation*}
		|\gamma+\alpha|\le \tau(E)-\tau(F),\quad 	|\beta-\gamma|\le \tau(F).
	\end{equation*} 
\end{corollary}

		\subsection{Some applications} \label{some_applications}  Our first set of applications regards  the situations described in Examples \ref{ex_model} and is based on the following observation. Given a positive definite measurable $n\times n$ matrix-valued function $w$ on $(a,b)$, we denote as in Section \ref{sec:examples} by $L^2((a,b), w)$ the corresponding weighted $L^2-$space consisting of $\bC^n-$ valued functions with
		$$\|f\|^2=\int_a^b\langle  wf(t),f(t)\rangle_{\bC^n}dt<\infty.$$
	As usual we identify functions $f,g$ with $wf=wg$ a.e..
		\begin{proposition}\label{interval_chain} Assume $H=L^2((a,b),w)$, $V$ is unicellular and for every $c\in (a,b)$ the subspace $$M_c=\{f\in L^2((a,b),w):~wf|[a,c]=0 \text{ \em a.e.}\},$$
			is invariant for $V$. Then every $V-$invariant subspace has the form $M_c$ for some $c\in (a,b)$.
\end{proposition}
		\begin{proof} The reason is that the set of subspaces in the statement is a maximal totally ordered set. If $M$ is $V-$invariant and nontrivial it must contain some $M_c$  and is contained in some $M_d$ for some $c,d\in (a,b),~c< d$ If 
			$$c_0=\sup\{c:~M_c\subset M\},\,d_0=\inf\{d:~M\subset M_d\},$$
			then clearly, $c_0=d_0$ and $M=M_{c_0}=M_{d_0}$	
\end{proof}
		
		1) Of course, this gives a quick proof of the classical result on invariant subspaces of the standard Volterra operator.

		2)  Assume $V$ is  Green's operator $V$  of a regular Schrödinger operator $D$ on $[a,b]$. From 2) of Example \ref{ex_model}, we know that the corresponding model is a de Branges space $\cH(E)$. It is also well known that the eigenfunctions of a regular Schrödinger operator have order $1/2$, see, for example, \cite{MR1943095} or \cite{MR216050}, that is,  the functions $\cH(E)$ have order at most $1/2$, so that $\tau(E)=0$ and by Corollary \ref{thm:1} the corresponding Green's operator is unicellular. In addition,  $H=L^2(a,b)$ and the subspaces $\cM_c$ in Proposition \ref{interval_chain} are $V-$invariant. Thus, all invariant subspaces of this operators are of the form $M_c,~c\in (a,b)$.

		3) The admissible pairs $(D,V)$ corresponding to a canonical system have a similar model (see 3) of Examples \ref{ex_model}), that is $\cW H=\cH(E)$ for some HB function $E$, in particular, $\alpha=0$. With the notations in \S \ref{sec:canonical_systems}, if $\bH$ is the Hamiltonian of the system then by Corollary \ref{thm:1}  the operator $V$ defined by 
		\begin{equation}\label{Green_canonical}
		Vf(x) = -\Omega \int_{0}^{x} \bH(t)f(t)dt, \quad f \in L^2((0,\ell);\bH),\,\text{or}\, f\in \mathscr{H}_\bH,
		\end{equation}
		is unicellular if and only if the type of $E$ is $0$. The exponential type of $E$ can be computed using the so called Krein-de Branges formula in terms of the associated Hamiltionian, see \cite{https://doi.org/10.48550/arxiv.1408.6022, MR0229011, MR3890099, MR133456} and \eqref{eq:krein_de_branges_formula_prelim} in Section \ref{sec:prelim}, and it is given by
		\begin{equation}\label{eq:KdB_formula}
		\tau(E) = \int_{0}^{\ell} \sqrt{\det{\left( \bH(x)\right)}}dx.
		\end{equation}
		Since $\bH$ is positive definite, this implies that $\tau(E)=0$  if and only if $\det{\bH(x)} = 0$, for almost every $x$. Proposition \ref{interval_chain} applies and we obtain the following result which actually contains the case of Green's operators  of regular Schr\"odinger operators considered above.
		\begin{corollary}\label{cor:greens_operator}
			Let $V$ be the Green's operator, given by equation \eqref{Green_canonical}, corresponding to a canonical system on $(0,\ell)$. 
			If $\det{\bH(x)} = 0$ a.e. then  the invariant   subspaces of $V$   are
			$$M_c=\{f\in L^2((0,\ell),\bH):~f|[0,c]=0 \text{ \em a.e.}\}, \quad c\in (0,\ell).$$
		\end{corollary} 	
		
		  If $\bH$ is not singular a.e. the situation is more complicated since $V$ might fail to be unicellular. 	  
		However,   by Theorem \ref{thm:1'} there always exist two rank one perturbations of $V$ which are  quasi-nilpotent and unicellular. These operators can be found with help  of Proposition \ref{quasi-nilpotent-ri}. Since $\alpha=0$, it yields  the equation
			\begin{equation*}
			\int_0^\ell \langle \bH(x)u_{\pm}(x), \phi_{\overline{z}}(x) \rangle_{\mathbb{C}^{2}} dx = \frac{1-e^{\pm i \tau(E)z}}{z}.
		\end{equation*}
		Although, $\tau(E)$ can be defined entirely in terms of the Hamiltonian via the Krein-de Branges formula \eqref{eq:KdB_formula} it seems to be a difficult task to find  explicit solutions of this equation.

	As pointed out in the introduction, the study of invariant subspaces of such Green's operators $V$ was suggested by C. Remling (\cite{MR3890099}, notes to Chapter 5).
	 One motivation is that understanding unicellularity in this context would lead to uniqueness theorems for Krein-de Branges canonical systems without appealing to the theory of de Branges spaces (given that said unicellularity was proven without de Branges spaces).
	 
	Let us now turn to  self-adjoint operators with removable spectrum as considered in \S \ref{sec:compact_A}, that is, 		
	compact self-adjoint  operators which possess  a  quasi-nilpotent rank one perturbation.  We want to show that under natural assumptions such operators have rank one perturbations which are both quasi-nilpotent and  unicellular. 
\begin{corollary}\label{thm:3}
	Let $A$ be an injective compact self-adjoint operator with simple spectrum on the separable Hilbert space $H$. Suppose there exists a rank one perturbation of $A$ of the form
	\begin{equation*}
		V_{x,y} u = Au + \langle u, x\rangle y, \quad u\in H,
	\end{equation*}
which is quasi-nilpotent and such that $y \notin \text{\rm Ran}(A)$. Then there exists $z \in H$, such that the rank one perturbation
	\begin{equation*}
		V_{z,y}u = Af + \langle u,z \rangle y, \quad u \in H,
	\end{equation*}
	is  quasi-nilpotent and  unicellular.
\end{corollary}
   \begin{proof}[Proof of Theorem \ref{thm:3}]
   	Recall that $A$ is an injective compact self-adjoint operator with removable spectrum. If $V_{x,y}$ is the quasi-nilpotent  rank one perturbation in the statement  with $y \notin \text{\rm Ran}(A)$, define $D$ by \eqref{def_D_sa_case} and use Proposition \ref{removable_spectrum} to conclude that $D,  V_{x,y}$ is an admissible pair with $Dy=0$. Then by 
   	Theorem \ref{thm:1'} $D$ has at least one unicellular quasi-nilpotent right inverse $V$ which must have the form $V_{z,y}$ for some $z\in H$.
   \end{proof}
\end{section}
    
\begin{section}{$D-$invariant subspaces} \label{sec:D_invariant}
	\subsection{The space $C^\infty(D)$}
  We now turn our attention to the operator $D$ and  study some of its invariant subspaces. Following the classical case we first introduce our analogue of $C^\infty$ using the simple observation below. For $n\in \mathbb{N}$ we define inductively $$\mathscr{D}(D^n)=\{x\in \mathscr{D}(D^{n-1}):~D^{n-1}x\in \mathscr{D}(D)\}.$$
  \begin{proposition}\label{dom_D} 
  We have $$\mathscr{D}(D^n)=\left\{\sum_{k=0}^{n-1}a_kV^k\phi_0 +V^nx;~a_k\in \mathbb{C}, 0\le k\le n-1,\, x\in H  \right\}.$$
  In particular, $\mathscr{D}(D^n)\subset\mathscr{D}(D^{n-1})$. Moreover, each $\mathscr{D}(D^n)$ is a Hilbert space with respect to the  norm
  \begin{equation}
  \label{V_n_norm}
 \|x\|_n^2=\sum_{k=0}^n\|D^kx\|_H^2, \end{equation}
 and the inclusion map from  $\mathscr{D}(D^{n+1})$	into $\mathscr{D}(D^n)$ is compact.
  	\end{proposition}
  
  \begin{proof} Since $DV=I$, $Dx=y$ implies $$x=Vy+a_0\phi_0,$$
 which gives the result for $n=1$. If we assume the assertion holds true for $n-1$ then any $y\in \mathscr{D}(D^{n-1})$ has the form 
 $y=p(V)\phi_0+V^{n-1}x$ with $p$ a polynomial of degree at most $n-2$. Then $D^{n-1}y\in \mathscr{D}(D)$, is equivalent to 
 $$x=c\phi_0+Vz\,\Rightarrow y=c\phi_0+Vp(V)\phi_0+V^nz,$$
 which concludes the proof of the first part. The remaining part is straightforward. For example, the compactness assertion follows directly from the compactness of $V$ on $H$.
\end{proof} 

\begin{remark}\label{compact_inclusion} 1)
The argument  actually shows that    $\mathscr{D}(D^n)$ is the maximal domain of $D^n$.\\
2) Obviously, $\mathscr{D}(D^n)$ is invariant for  any right inverse of $D$. If the right inverse in question is quasi-nilpotent on $H$, the same  will hold on $\mathscr{D}(D^n)$. Also, the compactness of the inclusion implies that the restriction of any right inverse of $D$ to this space is compact as well.
\end{remark}
 Recall that the eigenvectors of $D$ are $\phi_\lambda=(1-\lambda V)^{-1}\phi_0,~\lambda\in \mathbb{C}$ which obviously belong to $\mathscr{D}(D^n)$ for all $n\in \mathbb{N}$.
\begin{corollary}\label{V|D(D^n)} For every $n\in \mathbb{N}$, the map $\Phi(\lambda)=\phi_\lambda$ is an entire $\mathscr{D}(D^n)-$valued function with 
	$$\Phi^{(k)}(\lambda)=k!V^k(1-\lambda V)^{-k-1}\phi_0,$$
	and the linear span of $\{\Phi(\lambda)=\phi_\lambda:~\lambda\in \mathbb{C}\}$ is dense in $\mathscr{D}(D^n)$.
\end{corollary}
 \begin{proof} The first part follows directly from the simple observation  that $V |\mathscr{D}(D^n)$ is quasi-nilpotent, since by \eqref{V_n_norm} we have for $m>n$
 	$$\|V^m\|_n^2=\sum_{k=0}^n\|V^{m-k}x\|_H^2.$$ To verify the density of $\bigvee\{\phi_\lambda:~\lambda\in \mathbb{C}\}$, choose $\Lambda \subset \mathbb{R}$  such that $\left\{ \phi_{\lambda}:~ \lambda \in \Lambda\right\}$ is an orthogonal basis for $H$.  Then every $x\in H$ belongs to the closed linear span of these vectors, hence by a repeated application of $V$, there exists a polynomial $q$ such that $$V^nx+q(V)\phi_0\in \overline{\bigvee\left\{ \phi_{\lambda}:~ \lambda \in \Lambda\right\}}\subset \overline{\bigvee\left\{ \phi_{\lambda}:~ \lambda \in \mathbb{C}\right\}}.$$  
 Now by the definition of the derivative,  $$\frac1{k!}\phi^{(k)}(0)=V^k\phi_0\in \overline{\bigvee\left\{ \phi_{\lambda}:~ \lambda \in \mathbb{C}\right\}},\quad k\ge 0,$$  hence the set on the right contains $p(V)\phi_0$ for any polynomial $p$ and  Proposition \ref{dom_D} gives the desired result.
 	\end{proof}
 	
 Let us now turn to our analogue of $C^\infty$ defined as 
\begin{equation}\label{c_infty}
C^\infty(D)=\bigcap_{n\ge 1}\mathscr{D}(D^n).
\end{equation}
The space appears naturally as a projective limit of  the Hilbert spaces $\mathscr{D}(D^n)$ (see \cite{MR3154940}, Chapter 3). Alternatively, the  topology on $C^{\infty}(D)$ is given by the translation-invariant metric
  \begin{equation*}
    d(x,y) = \sum_{j=0}^{\infty} 2^{-j} \frac{\norm{D^{j}x-D^{j}y}_{H}}{1+\norm{D^{j}x-D^{j}y}_{H}}.
  \end{equation*}
and $C^{\infty}(D)$ is complete and separable with respect to this metric.
The operator  $D$ together with all its right inverses act continuously on $C^\infty(D)$. Moreover, $C^\infty(D)$ is a  dense subspace of $H$, since it contains all eigenvectors $\phi_\lambda,~\lambda\in \mathbb{C}$ of $D$.

 Recall from \cite{MR3154940} that any continuous
linear functional $\varphi$ on $C^{\infty}(D)$  can be written in the form 
\begin{equation}\label{clf}
	\varphi(x) = \sum_{j=0}^{n} \langle D^{j}x, y_{j} \rangle_{H},
\end{equation}
with $y_{j}  \in H\setminus\{0\},~0\le j\le n$ fixed. The representation is not unique and the smallest $n$ for which such a representation holds will be called the \emph{order} of $\varphi$.
Obviously,  $\varphi$ has order at most $n$ if and only if it extends to a continuous linear functional on $\mathscr{D}(D^n)$.
The space of these continuous functionals 
will be denoted by $(C^\infty(D))'$.

As an application we prove the following  proposition which collects basic density results that will be needed in the sequel.
\begin{proposition}\label{prop:density}
	We have
	\begin{enumerate}[(i)]
		\item $\left\{ \phi_{z} \right\}_{z \in \mathbb{C}}$ has dense linear span in $C^{\infty}(D)$,
		\item For any      
		$\lambda_{0} \in \mathbb{C}$, the set $\left\{[V(I-\lambda_{0}V)^{-1}]^j\phi_{\lambda_{0}}:~j\ge 0\right\}$ has dense linear span in $C^{\infty}(D)$.
	\end{enumerate}
\end{proposition} 

\begin{proof}
	(i) is a direct application of Corollary \ref{V|D(D^n)}. (ii)   If $\varphi$ is a continuous linear functional on $C^\infty(D)$ which annihilates   the set in the statement, by \eqref{clf} and another application	of Corollary \ref{V|D(D^n)} it follows that all derivatives of the entire function $\lambda\mapsto \varphi(\phi_\lambda)$ vanish at $\lambda_0$. Then the function vanishes identically and the result follows by (i). 
\end{proof}

 We shall extend our generalized Fourier transform $\mathcal{W}$ from Section \ref{sec:model} to the dual of $C^\infty(D)$ by
\begin{equation}\label{ext_Fourier}
	\mathcal{W}\varphi(z) = \varphi(\phi_{z})^{\#} = \sum_{j=0}^{N} \langle y_{j}, D^{j}\phi_{\overline{z}} \rangle = \sum_{j=0}^{N} z^{j} \langle y_{j}, \phi_{\overline{z}} \rangle.
\end{equation}

 The dual space $(C^\infty(D))'$ can be endowed with the weak-star topology, i.e. the smallest topology such that all maps $\varphi\mapsto\varphi(x),~x\in C^\infty(D)$ are continuous.

 We shall use repeatedly the following results. The first is an appropriate version of the Krein-Smulian theorem and can be found in \cite{MR3154940}, Chapter 6.
 \begin{proposition}\label{K_S} For $r>0$ let $$B(0,r)^\circ =\{\varphi\in (C^\infty(D))':~|\varphi(x)|\le 1~ \text{whenever } d(x,0) \le r)\}.$$  	
 	A convex set $C\subset (C^\infty(D))'$
 	is $w^*-$closed  if and only if $C\cap B(0,r)^\circ$ is $w^*-$closed for all $r> 0$.
 \end{proposition}
 Note that the separability of $C^\infty(D)$ implies that the $w^*-$topology on each polar set $B(0,r)^\circ$ is metrizable. For this reason, it is sufficient to prove the sequential closedness of the sets involved in the above proposition. The following observation is very useful in this sense.
 \begin{proposition}\label{bounded order} If  $(\varphi_n)$ converges weak-star to $\varphi$ in $(C^\infty(D))'$ then there exists $N\in\mathbb{N}$ such that the order of $\varphi_n$ is at most $N$ for all $n\ge 1$. In addition, the extensions of $\varphi_n$ to  $\mathscr{D}(D^N)$ converge weak-star as well.
 \end{proposition}\begin{proof}
 	By uniform boundedness (or Baire's theorem) there exists $\delta>0$ such that $|\varphi_n(x)|\le 1$ whenever $d(x,0)\le \delta$ and $n\ge 1$. Then if $N\in \mathbb{N}$ satisfies $\sum_{j>N}2^{-j}<\frac{\delta}{2}$, we have $$d(x,0) < \sum_{j=0}^{N} 2^{-j} \frac{\norm{D^{j}x}_{H}}{1+\norm{D^{j}x}_{H}}+\frac{\delta}{2} \le
 	\sum_{j=0}^N\|D^jx\|+\frac{\delta}{2}\le \sqrt{N}\|x\|_N+\frac{\delta}{2}.$$
 	This implies $|\varphi_n(x)|\le 1$ whenever $\|x\|_N\le\frac{\delta}{2\sqrt{N}}$ and $n\ge 1$, i.e. the order of $\varphi_n$ is at most $N$ for all $n\ge 1$. The sequence of restrictions of  $\varphi_n$ to  $\mathscr{D}(D^N)$ is uniformly bounded and converges on a dense subspace of $\mathscr{D}(D^N)$.
 \end{proof}

\subsection{Spectra of restrictions and annihilators of  invariant subspaces}\label{sec:subspace_examples}
  The main purpose of this section is to consider closed subspaces of $C^{\infty}(D)$ which are invariant under $D$. We shall denote by $\mathscr{J}$ the collection of all such subspaces. 
    
 As in the case of standard differentiation on intervals  contained in the real line, the structure of such subspaces is related to the spectrum of the restriction of the operator (see \cite{MR2419491}).
 We begin with some typical examples of invariant subspaces for $D$
 and spectra of the corresponding  restriction. To this end, it is useful to recall  (see, for example, the proof of Proposition \ref{dom_D}) that the solutions of $(D-\lambda I)x=y$ have the form 
 \begin{equation}\label{D_I)x=y}x=c\phi_\lambda + V(1-\lambda V)^{-1}y, \quad c\in \mathbb{C}.\end{equation}
 
 \begin{examples}\label{ex_spectra} {\rm 1)  Let $K$ be one of the spaces $H,~\mathscr{D}(D^n),~n\ge 1$, or $C^\infty(D)$, and let $M \subset K$ be a proper closed $V-$invariant subspace. Then 		\begin{equation*}
 			\mathcal{J}_{M} = \left\{ x \in C^{\infty}(D) : D^{j}x \in M \text{, for all } j \geq 0 \right\}. \end{equation*}
 belongs to 	$\mathscr{J}$ and satisfies 
 $\sigma(D|\mathcal{J}_M)=\emptyset$. To see this, note first that $\sigma(V|\mathcal{J}_M)=\{0\}$. 
 This is obvious when  $K=H$, or when $K=\mathscr{D}(D^n),~n\ge 1$. In the case $K=C^\infty(D)$  apply  the previous observation to the 
 closure of $M$ in $\mathscr{D}(D^n),~n\ge 1$ and use the fact that if $y\in C^\infty(D)$ the solution $x$ of $(\lambda V-I)x=y$ belongs to $C^\infty(D)$ as well. Thus,
 	$M$ is invariant for $(V-\mu)(1-\lambda V)^{-1},~\lambda,\mu\in \mathbb{C}$, in particular,  it cannot contain any $\phi_\lambda,~\lambda\in \mathbb{C}$. Otherwise, since $\phi_z=(I-\lambda V)(I-zV)^{-1}\phi_\lambda$, $M$ would contain all such elements and equal $H$.  This shows that  		 
 		 	and if $y\in \mathcal{J}_M$, the unique solution $x\in \mathcal{J}_M$ of the equation $(D-\lambda I)x=y$ is given by $$x= V(1-\lambda V)^{-1}y\in \mathcal{J}_M,$$
 		 	and the claim follows. \\ We also note that these assertions continue to if hold $V$ is replaced by any other quasi-nilpotent  right inverse of $D$.
 		 	
  2) The spectrum of $D$ on the whole space $C^\infty(D)$ equals $\mathbb{C}$. In the classical case, i.e. $D=-i\frac{d}{dx}$ on $C^\infty(a,b)$, there are also smaller $D-$invariant subspaces with this property, for example
 $$\mathcal{J}=\{f\in C^\infty(a,b):~f|I\cup J=0\},$$
 where $I,J$ are two disjoint closed subintervals of $(a,b)$.\\
 3) Let $A$ be a self-adjoint restriction of $D$ with compact resolvent. Then there exists an orthogonal basis of eigenvectors $\left\{ \phi_{\lambda} \right\}_{\lambda \in \sigma(A)}$, where $\sigma(A) \subset \mathbb{R}$ is discrete. \\ For $\emptyset\neq\Lambda\subset\sigma(A)$, let $\mathcal{J}_\Lambda$ be the closed span of $\{\phi_\lambda:~\lambda\in \Lambda\}$ in $C^\infty(D)$. This space is $D-$invariant since $D$ is continuous on $C^\infty(D)$, and  for  every finite linear  combination $$x=\sum_{j=1}^ma_j\phi_{\lambda_j}, \quad a_j\in \mathbb{C},\,\lambda_j\in \Lambda,$$
 we have $Dx\in \mathcal{J}_\Lambda$. \\We claim that $\sigma(D|\mathcal{J}_\Lambda)=\Lambda.$ Indeed, it is easy to verify that for $\mu\in \mathbb{C}\setminus \Lambda$, the linear map $R_\mu$ defined on the span of $\{\phi_\lambda:~\lambda\in \Lambda\}$ by 
$$R_\mu\left( \sum_{j=1}^ma_j\phi_{\lambda_j}\right) =\sum_{j=1}^m\frac{a_j}{\lambda_j-\mu}\phi_{\lambda_j} \quad a_j\in \mathbb{C},\,\lambda_j\in \Lambda,$$
 extends continuously to $\mathcal{J}_\Lambda$ and is an inverse of $(D-\mu I)|\mathcal{J}_\Lambda$.}

\end{examples}
 In the classical case when $D=-i\frac{d}{dx}$ on some interval on the real line, the three examples above describe completely the types of spectra that can occur for the restriction to an invariant subspace; $\mathbb{C},~\emptyset$ or a discrete subset of a complex plane (see \cite{MR2419491}). It  turns out that this result continues to hold in the general case.

 \begin{theorem}\label{thm:spectrum}
 	For each $\mathcal{J} \in \mathscr{J}$ the spectrum $\sigma(D|\mathcal{J})$ is either the whole complex plane, $\mathbb{C}$, or the discrete (possibly void) set
 	\begin{equation*}\sigma(D|\mathcal{J}) = \left\{ \lambda \in \mathbb{C} : \phi_{\lambda} \in \mathcal{J} \right\}.
 	\end{equation*}
 	If $\sigma(D|\mathcal{J}) \neq \mathbb{C},\emptyset$ then each $\lambda \in \sigma(D|\mathcal{J})$ has spectral multiplicity one.
 	\end{theorem}

 	\begin{proof} By Proposition \ref{prop:density} the set $\{\lambda:~\phi_\lambda\in \mathcal{J}\}$ must be discrete. Therefore, it will be sufficient to prove that if $\sigma(D|\mathcal{J}) \neq \mathbb{C}$, then for each $\mu\in \mathbb{C}$, $(D-\mu I)|\mathcal{J}$ is bijective whenever it is injective. The continuity of the inverse is automatic since $\mathcal{J}$ is Fr\'echet. To this end, assume $\lambda,\mu\in\mathbb{C}, ~\lambda\neq \mu$, are such that $(D-\lambda I)|\mathcal{J}$ has the continuous inverse $R:\mathcal{J}\mapsto
 	\mathcal{J}$, 	and $(D-\mu I)|\mathcal{J}$ is injective, i.e. $\phi_\mu\notin \mathcal{J}$. We need to show that 
 	$(D-\mu I)|\mathcal{J}$ is surjective.	 This is immediate in the case when $I+(\lambda-\mu)R$ is invertible on $\mathcal{J}$ since $(I+(\lambda-\mu)R)^{-1}x$ is a solution of $(D-\mu I)x=y$. 
 	
 	The remainder of the proof consists in removing the above assumption on $I+(\lambda-\mu)R$.	
 	Observe that from $(D-\lambda I)R=I$ it follows that  	
 	\begin{equation}\label{form_R}
 		Rx = \ell(x)\phi_{\lambda} + V(I-\lambda V)^{-1}x,
 		\quad x\in \mathcal{J},
 	\end{equation}	
 where $\ell$ extends to a continuous linear functional on $C^\infty(D)$. 
 Use \eqref{clf} to	choose $n_0\in \mathbb{N}$ 	such that $\ell$ is a continuous linear functional on $\mathscr{D}(D^{n_0})$, and for $n\ge n_0$ and let $\mathcal{J}_n$ be the closure of $\mathcal{J}$ in $\mathscr{D}(D^n)$. By Proposition \ref{dom_D} (see also 2) of Remark \ref{compact_inclusion}) it follows that $R$ extends to a compact operator on $\mathcal{J}_n,~n\ge n_0$. Moreover, if $n>1$ this extension satisfies $$(D-\lambda I)Rx=x,\quad x\in \mathcal{J}_n.$$
 Indeed, if $(x_k)$ is a sequence in $\mathcal{J}$ converging to $x\in \mathcal{J}_n$, then $Rx_k\to Rx$ in $\mathcal{J}_n$, and  since $D$ is continuous from $\mathscr{D}(D^n)$ to $\mathscr{D}(D^{n-1})$ we obtain
 $$(D-\lambda I)Rx=\lim_{k\to\infty}(D-\lambda I)Rx_k=\lim_{k\to\infty}x_k=x,$$
 where the above limits are considered in $\mathscr{D}(D^{n-1})$. Now  use the assumption that $\phi_\mu\notin \mathcal{J}$  to find
 $n_1\in \mathbb{N}$ such that $\phi_\mu\notin \mathcal{J}_n$ if $n\ge n_1$. Then from above  $I+(\lambda-\mu)R|\mathcal{J}_n$ is injective if $n> n_1$, since for $x\in \mathcal{J}_n$, $I+(\lambda-\mu)Rx=0$ implies
 $$0=(D-\lambda I)(I+(\lambda-\mu)R)x=	(D-\mu I)x,$$
 which implies $x=0$. But then, using the compactness of $R$, we conclude that $I+(\lambda-\mu)R|\mathcal{J}_n$ is invertible. Thus, for such $n$  we obtain exactly as above that for each $y\in \mathcal{J}$, the equation $(D-\mu I)x=y$ has the solution $$x_n=((I+(\lambda-\mu)R)|\mathcal{J}_n)^{-1}y.$$
 Finally, we observe that the solutions $x_n,x_m$ differ only by a multiple of $\phi_mu$ and since $m,n>n_1$ we conclude that $x_n=x_m=x\in \mathcal{J}$.
 
 	\end{proof}

Let us now turn to annihilators of nontrivial $D-$invariant subspaces. 
\begin{proposition}\label{lemma:annihilator}
	Let $\mathcal{J} \in \mathscr{J}$ and let $\varphi \in \mathcal{J}^{\perp}$ have order $N\ge 0$.
	\begin{enumerate}[(i)]
		\item If $\overline{\lambda} \notin \sigma(D|\mathcal{J})$ and $\mathcal{W}\varphi(\lambda) = 0$, then the continuous linear functional given by $$\varphi_{\lambda}(x) = \varphi(V(I-\overline{\lambda}V)^{-1}fx),$$ annihilates $\mathcal{J}$, if $N>0$ it has order $N-1$,  and satisfies
		\begin{equation*}
			\frac{\mathcal{W}\varphi(z)}{z-\lambda} = \mathcal{W}\varphi_{\lambda}(z).
		\end{equation*}
		\item If $\mathcal{J}^{\perp}$ contains a nonzero functional $\varphi$ such that $\mathcal{W}\varphi$ has infinitely many zeros in $\mathbb{C} \setminus  \overline{\sigma(D|\mathcal{J})}$. Then there exists a nonzero $y \in H$ which annihilates $\mathcal{J}$.
	\end{enumerate}
\end{proposition}
	\begin{proof}
		If $\overline{\lambda} \notin \sigma(D|\mathcal{J})$ 
	 the resolvent $R=(D-\overline{\lambda}I|\mathcal{J})^{-1}$  has the form
		\begin{equation*}
			Rx = \ell(x)\phi_{\overline{\lambda}} + V(I-\overline{\lambda}V)^{-1}x,
		\end{equation*} with $\ell$ a continuous linear functional on $C^\infty(D)$.
		Since $ \mathcal{W}\varphi(\lambda) = 0$ it follows that $$\varphi_{\lambda}(x) = \varphi(V(I-\overline{\lambda}V)^{-1}x)=\varphi(Rx) = 0,$$ for all $x \in \mathcal{J}$.  By the form of $\varphi$ it follows easily that $\varphi_{\lambda}$ has order $N-1$ if $N>0$. Moreover, 
				\begin{equation*}
			\mathcal{W}\varphi_{\lambda}(z) = \varphi(V(I-\overline{\lambda}V)^{-1}\phi_{z})^{\#} = \varphi((z-\overline{\lambda})^{-1}(\phi_{z} - \phi_{\overline{\lambda}}))^{\#} = \frac{\mathcal{W}\varphi(z)}{z-\lambda}.
		\end{equation*}
	(ii)  follows by a repeated application of (i).	
	\end{proof}

  \subsection{Residual subspaces and their annihilators}\label{annihilator_residual}
 As pointed out before, we are especially interested in $D-$invariant subspaces $\mathcal{J}$ with $\sigma(D|\mathcal{J})=\emptyset$. Such subspaces will be called \emph{residual}. They possess an alternative characterization based on the following simple observation.
 \begin{lemma}\label{p(D)-closed} If  $\mathcal{J} \in \mathscr{J}$ and $p$ is a polynomial then $p(D)\mathcal{J}$  is closed.
\end{lemma} 
 \begin{proof} It suffices to show that $(D-\lambda I)\mathcal{J},~\lambda\in \mathbb{C}$ is closed. Let $(x_k)$ be a sequence in $\mathcal{J}$ such that $(D-\lambda I) x_k\to y \in H$. Then there is a sequence $(c_k)$ in $\mathbb{C}$ such that 
 	\begin{equation}\label{c_k}V(I-\lambda V)^{-1}(D-\lambda I)x_k=x_k-c_k\phi_\lambda\to
 	V(I-\lambda V)^{-1}y,\end{equation}
 when $k\to\infty$. If $\phi_\lambda\in \mathcal{J}$, it follows that $(x_k)$ converges to $x\in \mathcal{J}$ and $y=(D-\lambda I)x\in (D-\lambda I)\mathcal{J}$. If $\phi_\lambda\notin \mathcal{J}$, let $\varphi$ be a continuous linear functional in $ \mathcal{J}^\perp$ with $\varphi(\phi_\lambda)=1$. Then applying $\varphi$ to both sides of \eqref{c_k} we obtain that $(c_k)$ converges and as above, we obtain that $(x_k)$ converges to $x\in \mathcal{J}$ and $y=(D-\lambda I)x\in (D-\lambda I)\mathcal{J}$.
\end{proof}
With the lemma in hand, the characterization of residual subspaces is as follows.
  \begin{proposition}\label{char-residual}
    A subspace $\mathcal{J} \in \mathscr{J}$ is residual if and only if $$\mathcal{J}= \bigcap\{p(D)\mathcal{J}:~p\text{ polynomial}\}.$$
\end{proposition}
    \begin{proof}    	 
      If $\sigma(D|\mathcal{J}) = \emptyset$ we have $(D-\lambda)\mathcal{J} = \mathcal{J}$ for all $\lambda$ and the equality in the statement holds. Conversely, if the equality holds, then $(D-\lambda I)\mathcal{J} = \mathcal{J}$ for each $\lambda$. Suppose for a contradiction that $D-\lambda_{0}I$ is not injective for some $\lambda_{0}$. Then for each $j \geq 0$ there exists a nonzero $x_{j} \in \mathcal{J}$, such that $(D-\lambda_{0})^{j}x_{j} = \phi_{\lambda_{0}}$. From this it follows that $V_{\lambda_{0}}^{j}\phi_{\lambda_{0}} \in \mathcal{J}$ for all $j \geq 0$, and $\mathcal{J} = C^{\infty}(D)$ by Proposition \ref{prop:density} (ii).
    \end{proof}
 Note that for every subspace $\mathcal{J} \in \mathscr{J}$ we can consider its \emph{residual part} 
$$\mathcal{J}_{res}= \bigcap\{p(D)\mathcal{J}:~p\text{ polynomial}\}.$$
By the above proposition   this is a closed residual subspace of $\mathcal{J}$

Let us now turn to annihilators of residual subspaces. It turns out that these spaces can be described using the generalized Fourier transform $\cW$. \\
To   state our result we  use the same notation as in Theorem \ref{theorem:space}, that is 
$\mathcal{W}H= e^{-i\alpha z}\mathcal{H}(E)$, with $\alpha\in \mathbb{R}$. 

\begin{theorem}\label{thm:annihilator}
	Let $\mathcal{J} \in \mathscr{J}$ be residual. Then:\\	
		\begin{enumerate}[(i)]
		\item If for all $0\ne\varphi\in \cJ^\perp$, $\cW\varphi$ has finitely many zeros in $\bC$,   there exists 
		 $\beta\in \mathbb{R}$ with $\lvert \beta + \alpha \rvert \leq \tau(E)$
		such that 
	\begin{equation*}
		\mathcal{W}\mathcal{J}^{\perp} =   \left\{p e^{i\beta z} :~p \text{ polynomial}\right\},	
\end{equation*}	
\item If there exists $0\ne\varphi\in \cJ^\perp$ such that $\cW\varphi$ has infinitely many zeros in $\bC$ then there exists a HB function $F$ such that  $\mathcal{H}(F)\in$ Chain$(\mathcal{H}(E))$  and a  compact interval $J\subset [-\alpha-\tau(E),-\alpha+\tau(E)]$, of length $2\tau(F)$ and with midpoint  $\sigma$
 such that	
	\begin{equation*}
		\mathcal{\cW}\mathcal{J}^{\perp} = \left\{pg:~g\in e^{i\sigma z}\mathcal{H}(F),\, p \text{ polynomial}\right\}\cup  \left\{q e^{i\beta z} : \beta\in J,\,q \text{ polynomial}\right\}.
	\end{equation*}
\item  The real number $\beta$ from (i), the interval $J$ as well as the set $\left\{pg:~g\in e^{i\sigma z}\mathcal{H}(F),\, p \text{ polynomial}\right\}$
from (ii) are uniquely determined by the residual subspace $\cJ$.\\
\item Every subspace of $(C^\infty(D))'$ of the form described in (i) or (ii) is the annihilator of a residual subspace of $C^\infty(D)$. 
		\end{enumerate}
	\end{theorem}

\begin{proof}  We begin with some simple observations regarding functionals $\varphi_\beta\in (C^\infty(D))'$ with $\cW\varphi_\beta(z)=e^{i\beta z}$ for some fixed $\beta\in \bC$ and $z\in \bC$.
	
	a) The parameter $\beta$ must be real and satisfy $\beta\in [-\alpha-\tau(E),-\alpha+\tau(E)]$, since by \eqref{clf} $$\cW\varphi_\beta(z)=\sum_{j=0}^n z^j\cW x_j(z),$$
	with $\cW x_j\in \cW H=e^{i\alpha z}\cH(E)$.
	
	b) For every $\beta\in [-\alpha-\tau(E),-\alpha+\tau(E)]$, $\varphi_\beta$ has order at most one. Indeed,   by a) and Proposition \ref{quasi-nilpotent-ri} the operator $V_\beta$ is a bounded right inverse of $D$ on $H$. If we fix 	
 $x_0\in H$ with $\cW x_0(0)=1$ we have 
 	$$ \cW x_0-e^{i\beta z}=\langle V^*_\beta x_0, D\phi_{\overline{z}}\rangle.$$
Then by the density of $\{\phi_z: z\in \bC\}$ in $C^\infty(D)$ we obtain
	$$\varphi_\beta(x)=\langle x, x_0\rangle-\langle Dx,  V^*_\beta x_0\rangle,\quad x\in C^\infty(D).$$
Even if this will only be used later,  note that since $\|V_\beta\|$ are uniformly bounded in $\beta$, the same holds true for the norms of $\varphi_\beta$ as functionals on $\mathscr{D}(D)$.	
Now let $\cJ$ be residual and $\varphi\in \cJ^\perp$ be such that the entire function of finite exponential type $\cW\varphi$ has finitely many zeros in $\bC$.	Then 
 $\mathcal{W}\varphi(z)=p(z)e^{i\beta z},\quad z\in \mathbb{C},$
	where $p$ is a polynomial and $\beta\in \mathbb{C}.$ By a repeated application of  Proposition \ref{lemma:annihilator} (i)  together with observation a) above,	it follows that 
	$\beta\in [-\alpha-\tau(E),-\alpha+\tau(E)]$, $\varphi_\beta\in \cJ^\perp$ and $\varphi=\varphi_\beta\circ p(D)$. 
	
	If all 
$\varphi\in \cJ^\perp$ satisfy this condition, that is, under the assumption in (i),  $\beta$ is uniquely determined by $\cJ$. Indeed, if  $\varphi_\beta,\varphi\gamma\in \cJ^\perp$ then $\cW\varphi_\beta-\cW\varphi_\gamma$ either has infinitely many zeros in $\bC$, or $\beta=\gamma$. This proves (i) and the corresponding uniqueness assertion in (iii).

Let us now assume that there exists $0\ne \varphi\in \cJ^\perp$ such that $\cW\varphi$ has infinitely many zeros. Another repeated application of  Proposition \ref{lemma:annihilator} (i)  shows that $\cJ^\perp$ contains nonzero functionals of order zero, i.e. the closure $\cJ_H$ of $\cJ$ in $H$ is strictly contained in $H$. 	
	Then $$\mathcal{N}=\{\cW y:~y\in \mathcal{J}_H^\perp\},$$
	is a nonzero nearly invariant subspace of $\cW H$ without common zeros. By Corollary \ref{subspace:deBranges} we have  $$\mathcal{N}=e^{i\sigma z}\cH(F),$$
	with $\mathcal{H}(F)\in$ Chain$(\mathcal{H}(E))$ and $$|\sigma+\alpha|\le \tau(E)-\tau(F),$$
Thus, the first set displayed in (ii) coincides with the set of $\varphi\in \cJ^\perp$ such that $\cW\varphi$ has infinitely many zeros and as such, it is uniquely determined by $\cJ$. The remaining set of functionals in $\cJ^\perp$ is, of course, uniquely determined by $\cJ$ as well.  It consists of functionals $\varphi\in \cJ^\perp$ such that $\cW\varphi$ has finitely many zeros. By the discussion at the beginning of the proof $\cW\varphi=qe^{i\beta z}$, where $q$ is a polynomial and $\beta\in [-\alpha-\tau(E),-\alpha+\tau(E)]$. Note that since $\cW\cJ^\perp$ is invariant under multiplication by polynomials, $\varphi_\beta\in \cJ^\perp$ whenever $e^{i(\beta-\sigma) z}$ is associated to $\cH(F)$, that is, whenever $|\beta-\sigma|\le\tau(F)$. Conversely, if $\varphi_\gamma\in \cJ^\perp$ then $\cW(\varphi_\gamma-\varphi\sigma)=e^{i\gamma z}-e^{i\sigma z}$ has infinitely many zeros in $\bC$. By observation b) above this immediately implies that $e^{i(\gamma-\sigma) z}$  is associated to $\cH(F)$, hence $|\gamma-\sigma|\le\tau(F)$. 
The proof of (i)-(iii) is now complete.\\ To see (iv) we need to show that any subspace described in (i) or (ii) is weak-star closed. By Proposition \ref{K_S}  we need to show that if $(\varphi_n)$ is a weak-star convergent sequence in such a subspace and also in $B^\circ(0,r)$ for some $r>0$, then its weak-star limit $\varphi$ lies in the subspace as well.\\
Assume first that for infinitely many $n$ we have 
$$\cW\varphi_n(z)=p_n(z)e^{i\beta_n z},$$
with $p_n$ polynomials and either $\beta_n=\beta$, or $\beta_n\in J$ for all such $n$. Then by Proposition \ref{bounded order}	it follows that the degrees of $p_n$ are bounded by $N$ for all $n$ and by passing to a subsequence, we may assume that $\beta_n\to \beta_0$ in the second  case. From the pointwise convergence of $p_ne^{i\beta_n z}$ we infer that $(p_n)$ is pointwise convergent to a polynomial of degree at most $N$ and the assertion follows. 

The remaining case is when there exists $n_0\in \mathbb{N}$ such that $\cW\varphi_n$ has infinitely many zeros for all $n\ge n_0$ and converges weak-star to $\varphi$.  In this case, the subspace in question has the form displayed in (ii) with some HB function $F$ and some interval $J=[\sigma-\tau(F),\sigma+\tau(F)]$.
Again by	Proposition \ref{bounded order},  there exists $N\in \mathbb{N}$ such that each $\varphi_n$ has order at most $N$, 
hence for $n\ge n_0$ 
we have $\cW\varphi_n(z)=f_n(z)\sum_{k=0}^N p_{kn}z^k,$
with $f_n\in e^{i\sigma z}\cH(F)$, or equivalently,
$$\varphi_n(x)=\sum_{k=0}^Np_{kn}\langle D^k x, y_n\rangle_H,$$
with $\cW y_n=f_n\in  e^{i\sigma z}\cH(F)$. Fix $\beta\in J$ and write
$$\varphi_n\circ V_\beta^N(x)=\left\langle x, \sum_{k=0}^Np_{kn} (V_\beta^*)^{N-k}y_n\right\rangle_H.$$
Since  $e^{i\sigma z}\cH(F)$ is invariant for $V_\beta^*$, this implies $\cW\varphi_n\circ V_\beta^N\in  e^{i\sigma z}\cH(F)$.
If $(\varphi_n)$ converges weak-star to $\varphi$ in $(C^\infty(D))'$, then by Proposition \ref{bounded order} we obtain that $\cW\varphi\circ V_\beta^N\in  e^{i\sigma z}\cH(F)$ as well. By a direct computation we obtain 
$$\cW\varphi\circ V_\beta^N =\frac{\cW\varphi-q e^{i\beta z}}{z^N}$$
with  $q$ a polynomial, that is,
$$\cW\varphi=\cW\varphi\circ V_\beta^N\circ D^N =z^N\cW\varphi\circ V_\beta^N =\cW\varphi-q e^{i\beta z}$$
and the proof of (iv) follows.
\end{proof}		

The singleton $\{\beta\}$  from (i), or the interval $J$ from (ii) which may reduce to a point as well, will be called \emph{the residual interval} of the residual subspace  $\cJ$.

 A somewhat  surprising aspect is that the first set displayed in (ii) does not necessarily determine the de Branges space $\mathcal{H}(F)$ uniquely. This is a more subtle matter which we now address. The issue is related to the well known fact that (see \cite{MR0229011}) in de Branges spaces, $\mathscr{D}(M_z)$, the domain of the operator  $M_z$ of multiplication by the independent variable may fail to be  dense in the space. In order to state our result we recall the correspondence between de Branges spaces and canonical systems, respectively Hamiltonians described in Section \ref{sec:prelim}. Let us denote by $\bH_E$ the Hamiltonian associated to $\cH(E)$ that way.
 \begin{proposition}\label{non_unique} Let $E$ be a regular HB function. The following are equivalent:\\
 	(i) There exist HB functions $F,G$ with $\cH(F),\cH(G)\in \text{\rm Chain}(\cH(E))$ with
 	$$\left\{ pg : g \in 
 			\mathcal{H}(F), p \text{ polynomial} \right\} = \left\{ pg : g \in \mathcal{H}(G), p \text{ polynomial} \right\},
 		$$
 	(ii) There exists a HB function $G$ with $\cH(G)\in \text{\rm Chain}(\cH(E))$ such that $\mathscr{D}(M_z|\cH(G))$ is not dense in $\cH(G)$,\\
 	(iii) $\bH_E$ has singular intervals. 
 \end{proposition}
\begin{proof} 
	(ii) $\Leftrightarrow$ (iii) is Lemma 7 in \cite{MR4507623}. If (ii) holds use the axiomatic definition of de Branges spaces to conclude that $\overline{\mathscr{D}(M_z|\cH(G))}$ is a de Branges subspace of $\cH(G)$, hence it belongs to $\text{Chain}(\cH(E))$ 
and  the equality in (i) holds with $\cH(F)=\overline{\mathscr{D}(M_z|\cH(G))}$.
Indeed, since
$$\frac{f-(f/g)(\la)g}{z-\la}\in \cH(F),\quad f,g\in \cH(G),\, g(\la)\ne 0,$$
it follows that $f,g\in \left\{ pg : g \in 
\mathcal{H}(F), p \text{ polynomial} \right\}$ which proves the claim. Conversely, if (i) holds we can assume without loss that $\cH(F)\subset\cH(G)$.
Let $f \in \mathcal{H}(G)$ with $f \notin \mathcal{H}(F)$. By assumption there exists a polynomial $p$ and $g \in \mathcal{H}(F)$, such that $f = gp$. Hence $f/p = g \in \mathcal{H}(F)$. If we choose the polynomial $p$ to have minimal degree, then $zf/p \notin \mathcal{H}(F)$. Thus, if $L$ denotes  the backward shift on $\cH(E)$, \begin{equation*}
	\mathcal{H}(F) \subsetneq \left\{ f \in \mathcal{H}(G) : Lf \in \mathcal{H}(F)\right\}.
\end{equation*}
Using again the axiomatic characterization of de Branges spaces it follows easily that the right hand side is a  de Branges subspace of $\cH(E)$ which we denote  by $\cH(G_1)$. If $g_0\cH(G_1)\cap \cH(F)^\perp\setminus\{0\}$ and $h\in \mathscr{D}(M_z|\cH(G_1))$, then $$\langle h,g_0\rangle=\langle Lzh, g_0\rangle =0,$$
because $Lzh\in \cH(F)$. This shows that $\mathscr{D}(M_z|\cH(G_1))$ is not dense in $\cH(G_1)$ and the proof is complete.		
		\end{proof}

 \subsection{General form of residual subspaces}
 A direct application of Theorem \ref{thm:annihilator} is the following description of residual subspaces which is a refined version of Theorem \ref{1st_residual}. The result also shows that the situation described in 1) of Examples \ref{ex_spectra},
 $$\cJ=\cJ_M=\{x\in C^\infty(D):~D^jx\in M\},$$  is generic. Here
 $M\subset H$, or $M\subset \mathscr{D}(D^n),~n\ge 1$, or $C^\infty(D)$ is a closed $V_\beta-$invariant subspace for some $\beta$ with $|\beta+\alpha|\le \tau(E)$. Recall that given a residual subspace $\cJ \subset C^\infty(D)$, we have denoted by $\cJ_H$ its closure in $H$
 \begin{theorem}\label{residual_general} If $\cJ \subset C^\infty(D)$ is residual then $$\cJ=\cJ_M,$$
 	where $M=V_{\beta}C^\infty(D)$ for some $\beta\in [-\alpha-\tau(E),-\alpha+\tau(E)]$, if $\cJ_H=H$ and $M=\cJ_H$ if $\cJ_H\ne H$. In this case, with the notations in Theorem \ref{thm:annihilator} we have  $J_H=\cH(F)^\perp$  which is $V_\beta-$invariant for all $\beta\in I$.	In all cases, for $\beta$ as above and $\lambda\in \bC$ we have $((D-\lambda I)|\cJ)^{-1}=V_\beta(1-\lambda V_\beta)^{-1}$.
 	\end{theorem}
 The proof is based on the following simple observation.
 \begin{lemma}\label{volterra_range}
 If $|\beta+\alpha|\le\tau(E)$ then $V_\beta C^\infty(D)$ is closed in $C^\infty(D)$ and every $\varphi\in (V_\beta C^\infty(D))^\perp$ satisfies
 $$\cW\varphi(z)=ce^{i\beta z},$$
 for some fixed $c\in \bC$ and all $z\in \bC$.
\end{lemma}
 \begin{proof} The fact that $V_\beta C^\infty(D)$ is closed in $C^\infty(D)$ follows immediately since both operators $D,V_\beta$ are continuous on this space and satisfy $DV_\beta=I$. To see the second part, let $\varphi\in (V_\beta C^\infty(D))^\perp$ and write
 	$$\varphi(x)=\sum_{j=0}^n \langle D^jx,y_j\rangle,\quad x\in C^\infty(D).$$  From $\varphi(V_\beta\phi_{\overline{\lambda}})=0$, $\lambda\in \bC$ we obtain for all $\lambda\in \bC$
 	$$0=\frac{\cW y_0(\lambda)-\cW y_0(0)e^{i\beta\lambda}}{\lambda}+\sum_{j=1}^n\lambda^{j-1}\cW y_j(\lambda),$$
 	which implies $\cW\varphi(\lambda)=\cW y_0(0)e^{i\beta\lambda},~\lambda\in \bC$. 
\end{proof}
 
\begin{proof}[Proof of Theorem \ref{residual_general}] If $\cJ$ is dense in $H$ then the result follows immediately from  Theorem \ref{thm:annihilator} (i) and Lemma \ref{volterra_range}. The $V_\beta-$ invariance, as well as the formula for the resolvent of $D|\cJ$ are straightforward.
		
	If $\cJ_H\ne H$, use Theorem \ref{thm:annihilator} (ii) and the notations therein to conclude that the image by $\cW$ of the orthogonal complement of $\cJ_H$ equals $e^{i\sigma z}\cH(F)$ which is $\cW V_\beta^*\cW^{-1}-$invariant for all $\beta\in \mathbb{R}$ such that $e^{i(\beta-\sigma) z}$ is associated to $\cH(F)$, or equivalently for all $\beta\in J$. Then $\cJ_H$ is $V_\beta-$invariant for all such $\beta$. In particular, it makes sense to consider the residual subspace $\cJ_{\cJ_H}$ of $C^\infty(D)$ which obviously contains $\cJ$.
	To see the reverse inclusion observe  that  the functionals of order zero in the annihilators of $\cJ$ and $\cJ_{\cJ_H}$ coincide with the orthogonal complement of $\cJ_H$ in $H$.  Thus, by $D-$invariance, $\cW\cJ_{\cJ_H}^\perp$  contains all polynomial multiples of functions in  $e^{i\sigma z}\cH(F)$ together with all polynomial multiples of the exponential functions  $e^{i\beta z}$,  such that $e^{i(\beta-\sigma) z}$ is associated to $\cH(F)$, or equivalently the set 
	$\{pe^{i\beta z}:~\beta\in J,~p\text{ polynomial}\}$. Consequently,  by Theorem \ref{thm:annihilator} (ii) $\cW\cJ_{\cJ_H}^\perp$ contains $\cW\cJ^\perp$ and the equality 
	$\cJ=\cJ_{\cJ_H}$ follows. Again,  the $V_\beta$-invariance, for $\beta\in J$ and the formula for the resolvent of $D|\cJ$ are straightforward.
	\end{proof}
 
 We should point out that in most cases the examples of residual subspaces revealed by the theorem are not trivial. This happens whenever  the interval $J$ given in Theorem \ref{thm:annihilator} (ii) is strictly contained in $[-\alpha-\tau(E),-\alpha+\tau(E)]$. Even if $\tau(E)=0$ and all these intervals reduce to a point, the space $J$ is nontrivial whenever the assumption in Theorem \ref{thm:annihilator} (i) holds, or when the space $\cH(F)$ from (ii) is strictly contained in $\cH(E)$.
 
 \subsection{Some special cases} 
 It is interesting to compare the description of residual subspaces in this general context with 
 the classical case when $D=-i\frac{d}{dx}$ on $C^\infty(a,b)$. Recall from \cite{MR2419491}
 (see also the Introduction) that in this case the residual subspaces consist of functions which vanish on a fixed compact subinterval of $(a,b)$. The interval may reduce to a point in which case we require that all derivatives of the functions in the subspace vanish at that point as well. In order to give an interpretation of this result in the general case considered here, we need to find natural  analogues to zeros of $C^\infty-$functions.  
 Theorem \ref{residual_general} suggest the following two alternatives:\\
 1)  $x\in C^\infty(D)$ ''vanishes'' at  $\beta\in  [-\alpha-\tau(E),-\alpha+\tau(E)]$  if $\varphi_\beta(x)=0$, or equivalently, by Lemma \ref{volterra_range}, $x\in V_\beta C^\infty(D)$. Recall that $\varphi_\beta$ is defined by $\cW\varphi_\beta(z)=e^{i\beta z}$.\\
 2) $x\in C^\infty(D)$ ''vanishes'' on the compact interval  $ J\subset [-\alpha-\tau(E),-\alpha+\tau(E)]$ if for all $n\ge 0$, $D^nx$ belongs to the $V_\beta-$ invariant subspace $(e^{i\sigma z}\cH(F))^\perp$, where $\sigma$ is the midpoint of $J$,  $2\tau(F)$ is its length and $\beta$ is an arbitrary point in $J$. 
  
  The second alternative is not ambiguous by the uniqueness statement in Theorem \ref{thm:annihilator} (iii). However, the condition is 
  less intuitive than, for example, the obvious generalization of 1):\\
  2') $x\in C^\infty(D)$ ''vanishes'' on the compact interval  $J \subset [-\alpha-\tau(E),-\alpha+\tau(E)]$ if $x\in \cap_{\beta\in J} V_\beta C^\infty(D)$, or equivalently, $\varphi_\beta(x)=0, \beta\in J$.
  
  By  Theorem \ref{thm:annihilator} (ii) it follows that 2) implies 2') in general. In the classical case the two conditions coincide due to the special structure of Volterra invariant subspaces of $L^2(a,b)$.  
  The next result clarifies completely the condition 2') in the most general case.
   
  \begin{theorem}\label{zero-based} Let $J\subset [-\alpha-\tau(E),-\alpha+\tau(E)]$ be a compact interval with non-void interior. Then $$\cJ_0=\bigcap_{\beta\in J}V_\beta C^\infty(D),$$
  	 is the largest residual subspace whose residual interval equals $J$.
\end{theorem}
\begin{proof} It will be sufficient to prove that $\cJ_0$ is $D-$invariant and for each $\gamma\in [-\alpha-\tau(E),-\alpha+\tau(E)]\setminus I$, $\cJ_0$ is not contained in $V_\gamma C^\infty(D)$. \\Indeed, if the above assertions hold, for $\beta\in I$ and $x=V_\beta y\in \cJ_0$, $V_\beta Dx=V_\beta DV_\beta y=x$, that is $V_\beta=D^{-1}$. A similar argument shows that $\cJ_0$ is residual. The second assertion implies that the residual interval of $\cJ_0$ equals $I$. Finally,  $\cJ_0^\perp$ consists of the closed linear span of functionals of the form $$\varphi(x)=\varphi_\beta(D^nx),\quad n\ge 0,\, x\in C^\infty(D).$$
By Theorem \ref{thm:annihilator} (ii) $\cJ_0^\perp$ is contained in the annihilator of any residual $\cJ$ with residual interval $J$.

Let us now turn to the claims. Since $\cJ_0^\perp$ is the $w^*-$closure of the linear span of $\{\varphi_\beta:~\beta\in I\}$ in the dual of $C^\infty(D)$, to verify $D-$ invariance, it will be sufficient to prove that for fixed $x\in C^\infty(D)$, the function $u_x(\beta)=\varphi_\beta(x)$ is differentiable on $[-\alpha-\tau(E),-\alpha+\tau(E)]$ with 
\begin{equation}\label{diff_ux}\frac{d u_x}{d\beta}(\beta)=u_{Dx}(\beta)=\varphi_\beta(Dx).\end{equation}
To this end, recall from the observation b) in the proof of Theorem \ref{thm:annihilator} that each $\varphi_\beta$ has order at most one and that they form a bounded set of continuous functionals on $\mathscr{D}(D)$.

Now notice that the function $u_x$ is twice continuously differentiable in 
$\beta\in [-\alpha-\tau(E),-\alpha+\tau(E)]$ whenever $x$ is a finite linear combination of $\phi_z,~z\in \bC$ and \eqref{diff_ux} obviously holds, Then for such $x$ and $\beta<\beta'\in \beta\in [-\alpha-\tau(E),-\alpha+\tau(E)]$ we have
$$u_x(\beta')=u_x(\beta)+(\beta'-\beta)\frac{d u_x}{d\beta}(\beta)+
\int_\beta^{\beta'}\int_\beta^t\frac{d^2 u_x}{d\beta^2}(s)dsdt.$$
Using \eqref{diff_ux} and the above argument we arrive at the estimate
$$|\varphi_{\beta'}(x)-\varphi_\beta(x)-\varphi_\beta(Dx)|\le M(\beta'-\beta)\|x\|_3,$$
for $x$ in a dense subset of $C^\infty(D)$. This clearly implies \eqref{diff_ux} for all $x\in C^\infty(D)$, hence also the $D-$invariance of $\cJ_0$.

Finally, let $\gamma\in [-\alpha-\tau(E),-\alpha+\tau(E)]\setminus J$. Then $\text{length}(J)<2\tau(E)$ and by Lemma  \ref{lemma:subspace_of_given_type} there exists a HB function $F$ with $\cH(F) \in \text{Chain}(\cH(E))$, $2\tau(F)=\text{length}(J)$.
Obviously, if $\sigma$ denotes the midpoint of $J$, $e^{i(\gamma-\sigma) z}$ is not associated to $\cH(F)$, hence by Theorem \ref{thm:annihilator} $\varphi_\gamma$ does not belong to $\cJ_M^\perp$, where $M=(e^{i\sigma x}\cH(F))^\perp$. Since this subspace contains the weak-star closure of the linear span of $\{\varphi_\beta:~\beta\in J\}$, the second claim follows and the proof is complete.\end{proof}
  
 There are cases when the  subspace described in the theorem is the unique residual subspace with a given residual interval. For example, this happens when $\text{Chain}(\mathcal{H}(E))$ is \emph{thin} in the sense of \cite{MR3925104}. This means  that the map
 \begin{equation}\label{thin_chain}
 	\tau(t) = \tau(E_{t}),
 \end{equation}
 is injective (and surjective onto $(0,\tau(E)]$). This covers the classical case $D=-i\frac{d}{dx}$, where the chain consists of Paley-Wiener spaces corresponding to  nested intervals, but there are many other examples. For instance, in terms of the Hamiltonian of the associated canonical system, see \S \ref{sec:canonical_systems}, it is easy to decide if a de Branges chain is thin. Indeed, the de Branges chain is not thin if and only if $\det(H(x)) = 0$ on some interval. Chains of de Branges spaces are one of the main corner stones of the theory and have been studied extensively. For some recent developments in the theory of de Branges chains see \cite{MR4507623}.

  In such cases theorem \ref{residual_general} reads as follows.
 \begin{corollary}\label{cor:thin_chain}
 	Assume that $\text{\rm Chain}(\mathcal{H}(E))$ is thin. If $\cJ$ is a residual subspace then there exists a compact interval $J \subset [-\alpha-\tau(E), -\alpha+\tau(E)]$ such that $$\cJ=\bigcap_{\beta\in J}V_\beta C^\infty(D),$$
 	if $J$ has non-trivial interior and if $J=\{\beta\}$ then
 	$$\cJ=\cJ_{V_\beta C^\infty(D)}.$$
\end{corollary}
  The opposite situation occurs when $\tau(E)=0$.  Note that in this case $\alpha=0$ Theorem \ref{theorem:near_invariance} part 2, and the only quasi-nilpotent right inverse of $D$ on $H$ is $V=V_0$. The next result gives a complete description of  residual subspaces in this case.
  \begin{corollary}\label{tau=0} Assume that $\tau(E)=0$ and let $\cJ$ be a residual subspace of $C^\infty(D)$. If $\cJ$ is dense in $H$ then $\cJ=\cJ_{VC^\infty(D)}$ and if $\cJ$  is not dense in $H$ then there exist $\cH(F)\in \text{\rm Chain}(\cH(E)$ such that $\cJ=\cJ_{(\cH(F))^\perp}$. In particular, residual subspaces are totally ordered by inclusion.  
  \end{corollary}
 \begin{proof}
The number $\beta$ in Theorem \ref{thm:annihilator} (i) is zero and the interval $J$ in Theorem \ref{thm:annihilator} (ii) is $\{0\}$. Also, an application of this theorem gives that the set 
$$\{\cJ^\perp:~\cJ \text{ residual}\},$$ is totally ordered by inclusion. The result follows by an application of Theorem \ref{residual_general}. 
 \end{proof} 
A concrete example  is provided by  
 regular Schr\"odinger operators.
\begin{corollary}\label{schrodinger_last}
	Let $D = -d^{2}/dx + q$ be a regular Schrödinger operator on $[a,b]$ with any separated boundary condition at $a$. Let $\mathcal{J}$ be a residual subspace for $D$.  Then there exists $c\in [a,b)$ such that
	\begin{equation*}
		\mathcal{J} = \left\{ f \in C^{\infty}(D) : f([a,c]) = \left\{ 0 \right\} \right\}\text{, when } c > a,
	\end{equation*}
and if $c = a$
	\begin{equation*}
		\mathcal{J} = \left\{ f \in C^{\infty}(D) : D^{k}f(a) = \left(D^{k}f\right)'(a) = 0 \text{, for all } k \in \mathbb{N} \cup \left\{ 0 \right\} \right\}.
	\end{equation*} 
\end{corollary}
\end{section}
    
\bibliographystyle{abbrv}
\bibliography{invariant_subspaces_for_generalized_differentiation_and_Volterra_operators}
\end{document}